\documentclass[leqno]{amsart}
\usepackage{amsfonts,amsmath,enumerate,latexsym,verbatim,amscd,mathrsfs,color,array,dsfont,appendix}

\usepackage[latin1]{inputenc}
\newtheorem{theorem}{Theorem}[section]
\newtheorem{remark}[theorem]{Remark}
\newtheorem{proposition}[theorem]{Proposition}
\newtheorem{hypothesis}[theorem]{Hypothesis}
\newtheorem{lemma}[theorem]{Lemma}
\newtheorem{definition}[theorem]{Definition}

\usepackage[colorlinks=true]{hyperref}
\hypersetup{urlcolor=blue,linkcolor=black,citecolor=black,colorlinks=true}

\DeclareMathAlphabet{\mathpzc}{OT1}{pzc}{m}{it}
 \usepackage{amsmath,amsthm,amsfonts,amssymb}
\usepackage[mathscr]{euscript}
 \usepackage{hyperref} 
 \usepackage{mathrsfs} 
\usepackage{graphicx}
\usepackage{color}
  \usepackage{epsfig}

\begin{document}
 
\title{ Well-posedness for a coagulation multiple-fragmentation equation}
\thanks{Accepted for publication: July 2013.}
\thanks{AMS Subject Classifications:  45K05.} 
\date{}
\maketitle     
 
\vspace{ -1\baselineskip}

{\small
\begin{center}
{\sc Eduardo CEPEDA}   
 \\
 Laboratoire d'Analyse et de Math\'ematiques Appliqu\'ees, UMR 8050
 \\
  Universit\'e Paris-Est. 61, avenue du G\'en\'eral de Gaulle, 94010 Cr\'eteil C\'edex  \\[10pt]
 (Accepted in \textit{Differential and Integral Equations})  
\end{center}
}

\numberwithin{equation}{section}
\allowdisplaybreaks


 \smallskip

 \begin{quote}
\footnotesize
{\bf Abstract.}  
We consider a coagulation multiple-fragmentation equation, which
describes the concentration $c_t(x)$ of particles of 
mass $x \in (0,\infty)$ at the instant $t \geq 0$ in 
a model where fragmentation and coalescence phenomena
 occur. We study the existence and uniqueness of 
 measured-valued solutions to this equation for 
 homogeneous-like kernels of homogeneity parameter 
 $\lambda \in (0,1]$ and bounded fragmentation kernels,
  although a possibly infinite total fragmentation rate, 
 in particular an infinite number of fragments, 
 is considered. This work relies on the use of a
  Wasserstein-type distance, which has shown to
   be particularly well-adapted to coalescence 
   phenomena. It was introduced in previous
    works on coagulation and coalescence.

\end{quote}

\section{Introduction}
\setcounter{equation}{0}

The coagulation-fragmentation equation is a deterministic equation that models the evolution in time
of a system of a very big number of particles (mean-field description) undergoing coalescences and
fragmentations. The particles in the system grow and decrease due to successive mergers and
dislocations, each particle is fully identified by its mass $x\in(0,\infty)$, we do not consider its
position in space, its shape nor other geometrical properties. Examples of applications of these
models arise in polymers, aerosols and astronomy.

In these notes we are interested in the phenomena of coagulation and fragmentation at microscopic
scale, we will describe the evolution of the concentration of particles of mass $x$ in the
following way. On the one hand, the coalescence of two particles of mass $x$ and $y$ gives birth a
new one of mass $x+y$, $\{x,y\}\rightarrow x+y$ with a rate proportional to the coagulation kernel
$K(x,y)$. On the other hand, the fragmentation of a particle of mass $x$ gives birth a new set of
smaller particles $x \rightarrow \{\theta_1 x, \theta_2 x,\ldots\}$, where $\theta_i x$ represents
the fragments of $x$, with a rate proportional to $F(x)\beta(d\theta)$ and where $F:(0,\infty)
\rightarrow (0,\infty)$ and $\beta$ is a positive measure on the set $\Theta = \left\lbrace\theta =
(\theta_i)_{i\geq 1} : 1> \theta_1\geq \theta_2 \geq \ldots\geq 0 \right\rbrace$. This means that
the distribution of the ratios of daughter masses to parent mass is only determined by a function of
these ratios (and not by the parent mass). Denoting $c_t(x)$ the concentration of particles of mass
$x\in(0,\infty)$ at time $t$, the dynamics of $c$ is given by
\begin{align}
\notag
\partial_t c_t(x) & =   \frac{1}{2}\int_0^{x} K(y,x-y) c_t(y)c_t(x-y)\, dy - c_t(x)
\int_0^{\infty}K(x,y) c_t(y)\,dy\\
&   + \int_{\Theta} \Big [\sum_{i=1}^{\infty}
\frac{1}{\theta_i}F\left(\tfrac{x}{\theta_i}\right)c_t\left(\tfrac{x}{\theta_i}\right) -
F(x)c_t(x) \Big ]\beta(d\theta). 
\label{forteq2}
\end{align}
The fragmentation part of the model was first introduced by Bertoin \cite{Bertoin} and takes into
account an infinite measure $\beta$ and a mechanism of dislocation with a possibly infinite number
of fragments.

The macroscopic scale version of this model (wich is intrinsecally stochastic) is studied in Cepeda
\cite{CepedaSto}. We believe that a \textit{hydrodynamical limit} result concerning this two
settings is possible to obtain in the following way. Denoting by $\mu^n = \frac{1}{n}
\sum_{i\geq1}\delta_{m_i}$ the empirical measure associated to the system composed by
$(m_1,m_2,\ldots)$, then the Coalescence-Fragmentation process associated $(\mu^n_t)_{t\geq0}$
converges to the solution to equation (\ref{forteq2}). For a first result concerning convergence in
the case where $F\equiv 0$ see Norris \cite{Norris,Norris2} and Cepeda-Fournier
\cite{Cepeda_Fournier} for a explicit rate of convergence. Nevertheless, this is not the aim of
these notes.

In this paper we are mainly interested in a result of general well-posedness, this means, with the
less possible assumptions on $K$, $F$, $\beta$ and the initial condition. Our method is based on the
use of the following distance: for $\lambda\in(0,1]$ and $c,d$ two positive Radon measures such
that $\int_{0}^{\infty} x^{\lambda} (c+d)(dx) < \infty$, we set
\begin{equation*}
d_{\lambda}(c,d) = \int_{0}^{\infty} x^{\lambda-1} \left|c((x,\infty)) - d((x,\infty))\right| dx.
\end{equation*}

In this paper we extend the result in Fournier-Lauren\c{c}ot \cite{Well-Pdnss} concerning only
coagulation, and we show existence and uniqueness to (\ref{forteq2}) for a class of homogeneous-like
coagulation kernels and bounded fragmentation kernels, in the class of measures having a finite
moment of order the degree of homogeneity of the coagulation kernel. Unfortunately this method does
not extend to unbounded fragmentation kernels. Our assumptions on $F$ are not very restrictive for
small masses, since we do not ask to $F$ to be zero on a neighbourhood of $0$. On the other hand, we
control the big masses imposing to the fragmentation kernel to be bounded near infinity.
Nevertheless, we are able to consider infinite total fragmentation rates for all $x>0$.

We have chosen this model for the fragmentation since it is actually more tractable mathematically,
see Bertoin \cite{Bertoin,Bertoin2} and Haas \cite{Haas,Haas2} where the properties of the only
fragmentation model are extensively studied. Kolokoltsov \cite{Kolokoltsov2} shows in the discrete
case a hydrodynamical limit result for a different model than ours, namely he introduces a mass
exchange Markov process. An extensive study of the methods used by the author are given in the books
\cite{Kolokoltsov_Book,Kolokoltsov_Book2}. Finally, we refer to Eibeck-Wagner \cite{Eibeck_Wagner2}
where a different model is studied which is used to approach general nonlinear kinetic equations.

The paper is organized as follows: we introduce some notation, definitions and the result in
Sections \ref{Notations} and the proof is given in Section \ref{Proofs}, we compare our result to
those known to us in Section \ref{OtherForms}.

\section[Notation and Definitions]{The Coagulation multi-Fragmentation equation.- Notation,
Definitions and Result}\label{Notations}
\setcounter{equation}{0}
We first give some notation and definitions. We consider the set of non-negative Radon measures
$\mathcal M^+$ and for $\lambda \in \mathbb R$ and $c\in \mathcal M^+$, we set
\begin{equation}\label{Intro:Momentl}
M_{\lambda}(c) := \int_0^{\infty} x^{\lambda} c(dx),\hspace*{1cm} \mathcal M^+_{\lambda}
=\left\lbrace c \in \mathcal M^+, \,\,M_{\lambda}(c) < \infty \right\rbrace.
\end{equation}

\noindent Next, for $\lambda \in (0,1]$ we introduce the space $\mathcal H_{\lambda}$ of test
functions,
\begin{equation*}
\mathcal H_{\lambda} = \Big \lbrace \phi \in \mathcal C([0,\infty)) \textrm{ such that } \phi(0) = 0
\textrm{ and } \sup_{x\neq y} \frac{|\phi(x) - \phi(y)|}{|x-y|^{\lambda}} < \infty  \Big \rbrace.
\end{equation*}
Note that $\mathcal C_c^1((0,\infty)) \subset H_{\lambda}$. 

Here and below, we use the notation $x\wedge y := \min\{x,y\}$ and $x\vee y := \max \{x,y\}$ for
$(x,y) \in (0,\infty)^2$.

\begin{hypothesis}[Coagulation and Fragmentation Kernels]\label{Def_Kernels}
Consider $\lambda \in (0,1]$ and a symmetric coagulation kernel $K: (0,\infty) \times (0,\infty)
\rightarrow [0,\infty)$ i.e., $K(x,y) = K(y,x)$. Assume that $K$ is locally Lipschitz, more
precisely assume that it belongs to $W^{1,\infty}((\varepsilon,1/\varepsilon)^2)$ for every
$\varepsilon >0$ and that it satisfies
\begin{align}
K(x,y) & \leq   \kappa_0(x+y)^{\lambda},\label{Hyp_K1} \\
(x^{\lambda}\wedge y^{\lambda}) |\partial_x K(x,y)| & \leq   \kappa_1 x^{\lambda-1}
y^{\lambda},\label{Hyp_K2}
\end{align}
for all $(x,y) \in (0,\infty)^2$ and for some positive constants $\kappa_0$ and $\kappa_1$. Consider
also a fragmentation kernel $F:(0,\infty) \rightarrow [0,\infty)$ and assume that $F$ belongs to
$W^{1,\infty}((\varepsilon,1/\varepsilon))$ for every $\varepsilon >0$ and that it satisfies
\begin{align}
F(x) & \leq   \kappa_2, \label{Hyp_F1}\\
|F'(x)| & \leq    \kappa_3\, x^{-1}, \label{Hyp_F2}
\end{align}
for $x \in (0,\infty)$ and some positive constants $\kappa_2$ and $\kappa_3$. 
\end{hypothesis}
For example, the coagulation kernels listed below, taken from the mathematical and physical
literature, satisfy Hypothesis \ref{Def_Kernels}.
\begin{align*} 
K(x,y) & =   (x^\alpha + y^\alpha)^\beta  \\
& \textrm{with}\,\,\,
\alpha\in(0,\infty),\,\,\beta\in(0,\infty) \,\,\,\textrm{and}\,\,\,\lambda = \alpha\beta \in(0,1],
\\
K(x,y) & =   x^\alpha y^\beta + x^\beta y^\alpha \\
 & \textrm{with}\,\,\, 0 \leq \alpha \leq \beta \leq
1 \,\,\,\textrm{and}\,\,\,\lambda = \alpha+\beta \in(0,1], \\
K(x,y) & =  (xy)^{\alpha/2} (x + y)^{-\beta} 
\\
 & \textrm{with}\,\,\, \alpha \in (0,1],\,\,\,
\beta\in[0,\infty) \,\,\,\textrm{and}\,\,\,\lambda = \alpha-\beta \in(0,1], \\
K(x,y) & =   (x^\alpha+y^\alpha)^{\beta} |x^{\gamma} - y^{\gamma}| 
\\
 & \textrm{with}\,\,\, \alpha \in
(0,\infty),\,\,\, \beta\in(0,\infty),\,\,\,\gamma\in(0,1] \,\,\,\textrm{and}\,\,\,\lambda =
\alpha\beta+\gamma \in(0,1], \\
K(x,y) & =   (x+y)^{\lambda} e^{-\beta(x+y)^{-\alpha}} \\
 & \textrm{with}\,\,\, \alpha \in
(0,\infty),\,\,\, \beta\in(0,\infty), \,\,\,\textrm{and}\,\,\,\lambda \in(0,1].
\end{align*}
On the other hand, the following fragmentation kernels satisfy Hypothesis \ref{Def_Kernels}.
 
$
F(x) \equiv  1, 
$
 all non-negative function $F\in C^2(0,\infty),$ bounded, convex and
non-increasing, 
 all non-negative function $   F\in C^2(0,\infty),$ bounded, concave and
non-decreasing.

  We define the set of ratios by
$ \Theta = \left\lbrace \theta = (\theta_k)_{k\geq 1} : 1 > \theta_1 \geq \theta_2
\geq \ldots \geq 0\, \right\rbrace.
$
  
\begin{hypothesis}[The $\beta$ measure]\label{Hyp_beta}
We consider on $\Theta$ a measure $\beta(\cdot)$ and assume that it satisfies 
\begin{align}
&
\beta 
\Big ( \sum_{k\geq 1} \theta_k > 1 \Big )  =   0, \label{Hyp1_Beta} \\
& 
C_{\beta}^{\lambda} := \int_{\Theta}  \Big [\sum_{k\geq 2} \theta_k^{\lambda} +
(1-\theta_1)^{\lambda} \Big ]\beta(d\theta)  < 
 \infty, \hspace{5mm} \textrm{for some } \lambda\in
(0,1]. \label{Hyp2_Beta}
\end{align}
\end{hypothesis}

\begin{remark}\label{on_beta} 
\rm 
i) The property (\ref{Hyp1_Beta}) means that there is no gain of mass due to the dislocation of a
particle. Nevertheless, it does not exclude a loss of mass due to the dislocation of the particles.

 ii)
Note that under (\ref{Hyp1_Beta}) we have $\sum_{k\geq 1} \theta_k - 1\leq 0$ $\beta$-a.e.,
and since $\theta_k \in [0,1)$ for all $k\geq 1$, $\theta_k \leq \theta_k^{\lambda}$, we have
\begin{equation}\label{inebeta}
\left\lbrace
\begin{array}{ll}
1-\theta_1^{\lambda} \leq 1-\theta_1 \leq (1-\theta_1)^{\lambda},\,\,   \beta-a.e.,\\[2mm] 
\sum_{k \geq 1} \theta_k^{\lambda} - 1 = \sum_{k \geq 2} \theta_k^{\lambda} - (1 -
\theta_1^{\lambda}) \leq \sum_{k \geq 2} \theta_k^{\lambda},\,\, \beta-a.e.
\end{array} \right.
\end{equation} 
implying the following bounds:
\begin{equation}\label{Clambdabounds}\left\lbrace
\begin{array}{c}
\displaystyle\int_{\Theta} (1-\theta_1) \beta(d\theta) \leq C_{\beta}^{\lambda}, \,\,\,\displaystyle
\int_{\Theta} 
 \Big [\sum_{k\geq 2} \theta_k^{\lambda} + (1-\theta_1^{\lambda})  \Big ]
\beta(d\theta) \leq C_{\beta}^{\lambda}, \\ [4mm]
\displaystyle \int_{\Theta}  \Big ( \sum_{k \geq 1} \theta_k^{\lambda} -1  \Big )^+\beta(d\theta)
\leq C_{\beta}^{\lambda}.
\end{array} \right. 
\end{equation}
We point out that 
$$
\int_{\Theta} \Big | \sum_{k \geq 1} \theta_k^{\lambda} -1  \Big |\beta(d\theta)
\leq 2 C_{\beta}^{\lambda}
$$
 but when the term $\sum_{k \geq 1} \theta_k^{\lambda} - 1$ is negative
our calculations can be realized in a simpler way. We will thus use the positive bound given in the
last inequality. 
\end{remark}

The result of the deterministic framework depends strongly on the use of the distance which is
defined for $\lambda\in (0,1]$ and $c,\,d \in\mathcal M_{\lambda}^+$ (recall \ref{Intro:Momentl})
the distance
\begin{equation}\label{Intro:DistanceCoag}
d_{\lambda}(c,d) = \int_{0}^{\infty} x^{\lambda-1} 
 \Big |\int_x^{\infty}  (c(dy) -
d(dy)  ) \Big | dx.
\end{equation}

\begin{definition}[Weak solution to (\ref{forteq2})]\label{Def_solution}
Let $c^{in} \in M^+_{\lambda}$. A family $(c_t)_{t\geq 0} \subset \mathcal M^+$ is a
$(c^{in},K,F,\beta,\lambda)$-weak solution to $(\ref{forteq2})$ if $c_0 = c^{in}$,
\begin{equation*}
t\mapsto \int_0^{\infty} \phi(x) c_t(dx) \,\, \textrm{ is differentiable on } [0,\infty)
\end{equation*}
for each $\phi \in \mathcal H_{\lambda}$, and for every $t \in [0,\infty)$,
\begin{equation}\label{Def_finitemoment}
\sup_{s\in[0,t]} M_{\lambda}(c_s) < \infty,
\end{equation}  
and for all $\phi \in \mathcal H_{\lambda}$
\begin{align} \label{weakeq2}
\frac{d}{d t} \int_0^{\infty}\phi(x) c_t(dx) & =   \frac{1}{2}\int_0^{\infty}\int_0^{\infty} K(x,y)
(A\phi)(x,y) c_t(dx)c_t(dy)\\
&    + \int_0^{\infty}F(x)\int_{\Theta} (B\phi)(\theta,x)  \beta(d\theta) c_t(dx), \nonumber 
\end{align}
where the functions $(A\phi):(0,\infty)\times (0,\infty) \rightarrow \mathbb R$ and $(B\phi): \Theta
\times (0,\infty) \rightarrow \mathbb R$ are defined by
\begin{align}
(A\phi)(x,y)  & =   \phi(x+y) - \phi(x) - \phi(y), \label{operatorA} \\
(B\phi)(\theta,x) & =  \sum_{i=1}^{\infty} \phi(\theta_i x) - \phi(x).\label{operatorB} 
\end{align}
\end{definition}

This equation can be split into two parts, the first integral explains the evolution in time of the
system under coagulation and the second integral explains the behaviour of the system when
undergoing fragmentation and it corresponds to a growth in the number of particles of masses
$\theta_1 x$, $\theta_2 x$, $\ldots$, and to a decrease in the number of particles of mass $x$ as a
consequence of their fragmentation.

According to (\ref{Hyp_K1}), (\ref{Hyp_F1}), Lemma \ref{lemma31}. below, (\ref{Def_finitemoment})
and (\ref{Hyp2_Beta}), the integrals in (\ref{weakeq2}) are absolutely convergent and bounded with
respect to $t\in[0,s]$ for every $s\geq0$.

The main result reads as follows.

\begin{theorem}\label{Theorem}
Consider $\lambda \in (0,1]$ and $c^{in} \in \mathcal M_{\lambda}^+$. Assume that the coagulation
kernel $K$, the fragmentation kernel $F$ and the measure $\beta$ satisfy Hypotheses
$\ref{Def_Kernels} $ and $\ref{Hyp_beta}$ with the same $\lambda$.

  Then, there exists a unique $(c^{in},K,F,\beta,\lambda)$-weak solution to 
  $(\ref{forteq2}).$
\end{theorem}

It is important to note that the main interest of this result is that only one moment is asked to
the initial condition $c^{in}$. The assumptions on the coagulation kernel $K$ and the measure
$\beta$ are reasonable. Whereas the main limitation is that we need to assume that the fragmentation
kernel is bounded. It is also worth to point out that we have chosen to study this version of the
equation because of its easy physical intuition.

\section[Other formulations for Fragmentation]{Other formulations for the fragmentation
equation}\label{OtherForms}
\setcounter{equation}{0}
To enable us to compare our results to those obtained in other works, we discuss the relationships
between the various formulations. The first works (see \cite{Aldous,Laurencot,FournierGiet}) were
concentrated on the binary fragmentation where the particles dislocate only into two particles:

\smallskip

\noindent \textbf{Binary Model.}
Denoting $c_t(x)$ the concentration of particles of mass $x\in(0,\infty)$ at time $t$, the dynamics
of the fragmentation is given by the operator
$$
(\mathcal F_b c_t)(x)   =   \int_x^{\infty} F_b(x,y-x)c_t(y) dy - \frac{1}{2} c_t(x) \int_0^x
F_b(y,x-y) dy,
$$
for $(t,x)\in(0,\infty)^2$. The binary fragmentation kernel $F_b$ is also a symmetric function and
$F_b(x,y)$ is the rate of fragmentation of particles of mass $x + y$ into particles of masses $x$
and $y$.

Note that we can obtain the continuous coagulation 
binary-fragmentation equation, for example, by
considering $\beta$ with support in 
$\{\theta : \theta_1 + \theta_2 = 1 \}$ and $\beta(d\theta) =
h(\theta_1)\,d\theta_1 \delta_{\{ \theta_2 = 1-\theta_1\}}$,
 and setting $F_b(x,y) = \frac{2}{x+y}
F(x+y)h (\frac{x}{x+y} )$ where 
$h(\cdot)$ is a continuous function on $[0,1]$ and
symmetric at $1/2$. Under this framework,
 one can find some results of existence and uniqueness for
example in \cite{DubovskiStewart,StewartE,StewartU}.

\smallskip

\noindent
 \textbf{Multifragmentation Model.}
We can consider a version of the coagulation - multi fragmentation equation where the fragmentation
operator has the following representation; see \cite{Laurencot}:
\begin{equation*}
(\mathcal F_m c_t)(x) = \int_x^{\infty} F_m(y,x) c_t(y) dy - c_t(x)\int_0^x \frac{y}{x} F_m(x,y) dy,\end{equation*}
where $F_m(x,y)$ is the fragmentation kernel and explains the dislocation of a particle $x$ into
smaller particles $y$ and $x-y$. In the same spirit in
\cite{Banasiak,Banasiak_Lamb2,Banasiak_Lamb,Giri,GKW} is considered an equivalent representation of
the fragmentation operator
\begin{equation*}
(\mathcal F_m c_t)(x) = \int_0^x b(x,y) a(y) c_t(y) dy - a(x) c_t(x),
\end{equation*}
where $a(x) = \int_0^x \frac{y}{x} F_m(x,y) dy$ is the total rate of fragmentation of a particle of
mass $x$, and $b(x,y) = F_m(y,x)/a(y)$ is a non-negative function and represents the distribution
(probability) of particles of mass $x$ generated from particles of mass $y\geq x$. This operator
allows to consider a multi-fragmentation model in the following way, for each fragmentation of a
particle of mass $y$, the average number and mass of the fragments $x$ are, respectively
\begin{equation}\label{WeakAlternative}
N(y) = \int_0^y b(x,y) dx,\hspace{1cm} \textrm{and}\hspace{1cm} m(y) = \int_0^y x b(x,y) dy,
\end{equation}
and it is usually assumed that no mass is lost when a particle breaks up, that is, $\int_0^y x
b(x,y) dy = y$. In both the physics and mathematics literature, concerning the fragmentation
operator, particular attention has been paid to models with the following self-similar dynamic:
\begin{enumerate}
\item[$\bullet$] $S(x)=C x^\alpha$, for some constant $C>0$ and $\alpha\in\mathbb R$.
\item[$\bullet$] $b(x,y) = \frac{1}{x} h (\frac{y}{x} )$ with $\int_0^1 x h(x)dx = 1$.
\end{enumerate}
The main two reasons for this are that self-similar assumptions are relevant for applications and
that they are also more mathematically tractable. There is also a significant literature on
probabilistic models for the microscopic mechanism of fragmentation with a self-similar dynamic. We
refer to the book by Bertoin \cite{Bertoin_Book} for an overview and to \cite{EMR,Haas,Haas2}
for discussions of the relations between the probabilistic models and the above operator.

Remark that we express the rate of fragmentation of
 a particle of mass $x$ as the product
$F(x)\beta(d\theta)$. If we consider fragmentation
 kernels of the form $F_m(x,y) = F(x) \frac{1}{x} h (\frac{y}{x} )$, 
note that the rate of fragmentation of a particle of mass $x$ is $a(x) =
F(x)\int_0^1 h(\theta)d\theta$ which, under our assumptions, can be infinite for all $x$, and
denoting $\theta$ the fragments (\ref{WeakAlternative}) becomes
\begin{equation*}
(\mathcal F_m c_t)(x) = \int_0^1 
 \Big [\tfrac{1}{\theta^2} F\left(\tfrac{x}{\theta}\right)
c_t\left(\tfrac{x}{\theta}\right) - F(x) c_t(x) \Big ] h(\theta)d\theta.
\end{equation*}
Nevertheless, it is not clear the existence of a measure $h$ such that allow the identification 
\begin{equation*}
(\mathcal F_m c_t)(x) = \int_{\Theta}  \Big [\sum_{i=1}^{\infty}
\tfrac{1}{\theta_i}F\left(\tfrac{x}{\theta_i}\right)c_t\left(\tfrac{x}{\theta_i}\right) -
F(x)c_t(x) \Big ]\beta(d\theta),
 \end{equation*}
 which demands some properties to the measure $h$.  
 
On the one hand, one of the difficulties when working with the coagulation-fragmentation equation,
as stated in Banasiak-Lamb \cite{Banasiak_Lamb2}, is that the coagulation operator is not linear.
The authors used a compactness method, the method used constrains the authors (see
\cite{Laurencot,Giri,GKW}) to require some finite moments to the initial conditions, existence holds
in the functional set 
$$
X = \Big  \{f \in L^1(0,\infty): \int_0^{\infty} (1+x)|f(x)|dx<\infty  \Big \}
$$
 (and the
solutions are not measures), in \cite{Banasiak_Lamb} is required higher moments to treat different
fragmentation rates than those found in the other works. It is also needed to control the number of
fragments at each dislocation and $\beta$ must be integrable.

It is worth to point out that the method we use in this paper relies on a previous result on the
coagulation-only equation, which considers a particular well-adapted distance that allows to relax
the hypotheses on the initial condition. The coagulation-only ($F\equiv 0$) equation is known as
Smoluchowski's equation and it has been studied by several authors, Norris in \cite{Norris} gives
the first general well-posedness result and Fournier and Lauren\c cot \cite{Well-Pdnss} give a
result of existence and uniqueness of a measured-valued solution for a class of homogeneous-like
kernels. The fragmentation-only ($K\equiv 0$) equation has been studied in Bertoin \cite{Bertoin}
and Haas \cite{Haas}. In particular, in \cite{Bertoin} the self-similar fragmentations are
characterized using a fragmentation kernel of the type $F(x) = x^\alpha$ for $\alpha\in\mathbb R$
and where the particles may undergo multi-fragmentations.

The main aim of this paper is to extend this result to the case where fragmentation is added to the
process. We remark that the model is different and allows us to consider other features of the
fragmentation that previous models do not present. Namely, we allow the fragmentation to give an
infinity of fragments at each dislocation and the measure $\beta$ is not necessarily integrable. In
this sense, although we consider bounded fragmentation kernels, the total fragmentation rate can be
infinite for each $x\geq0$.

Roughly, in \cite{Laurencot}, an existence and uniqueness result is given for $K(x,y) =
r(x)r(y)+\alpha(x,y)$, where $\alpha\in\mathcal C([0,\infty)\times[0,\infty))$ is the dominant term
for the coagulation-fragmentation process since the kernel $F_m\in\mathcal
C([0,\infty)\times[0,\infty))$ is assumed to satisfy 
$$
F_m(x,y) \leq C(1+\max(x,r(x))) , \ \ \text{ for
$x,y\geq0$ and }
$$
$$
\int_0^x F_m(x,y) \leq \gamma(x) \max(x,r(x)) , 
\ \text{ for $x\geq0$ and
$\gamma:[0,\infty)\rightarrow[0,\infty)$}
$$
 with $\gamma(x)\underset{x\rightarrow
\infty}\longrightarrow0$.
In \cite{Giri,GKW} the coagulation kernel is assumed to satisfy 
$$
K(x,y) \leq C (1+x)^\mu (1+y)^\mu
$$
with $\mu\in[0,1)$, and $a(x) \leq C_1(1+x)^{a_1}$ and 
$$
\int_0^x (1+y)^{1+\nu} b(y,x) dy \leq
C_2(1+x)^{a_2} ,
$$
where $C_1$ and $C_2$ are positive constants and where $a_1+a_2\leq 1+\nu$ with
$1+\nu\in(0,1)$.
Finally, in \cite{Banasiak_Lamb} the authors consider $a(x)\leq C_1(1+x^\mu)$ and 
$$
\int_0^x
yb(y,x)dx\leq C_2 (1+x^\nu) ,
$$
 with $\mu,\nu \in [0,\infty)$ and where $C_1$ and $C_2$ are positive
constants. This result allows to consider stronger fragmentation rates requiring a stronger moment
for the initial condition.
 
\section{Proofs}\label{Proofs}
\setcounter{equation}{0}

We begin giving some properties of the operators $(A\phi)$ and $(B\phi)$ for $\phi \in \mathcal
H_{\lambda}$ which allow us to justify the weak formulation (\ref{weakeq2}).

\begin{lemma}\label{lemma31} Consider $\lambda \in (0,1]$, $\phi \in \mathcal H_{\lambda}$. Then
there exists $C_{\phi}$ depending on $\phi$, $\theta$ and $\lambda$ such that
$$
(x+y)^{\lambda} |(A\phi)(x,y)|  \leq   C_{\phi} (xy)^{\lambda}, \ \ \
|(B\phi)(\theta,x)|   \leq   C_{\phi} x^{\lambda} 
 \Big [\sum_{i\geq 2}\theta^{\lambda}_i +
(1-\theta_1)^{\lambda} \Big ],
$$
for all $(x,y)\in(0,\infty)^2$  and for all $\theta \in \Theta$.
\end{lemma}

\noindent
{\bf Prof of Lemma \ref{lemma31}.}
For $(A\phi)$ we recall \cite[Lemma 3.1]{Well-Pdnss}. Next, consider $\lambda \in (0,1]$ and $\phi
\in \mathcal H_{\lambda}$ then, since $\phi(0) = 0$,
\begin{align*}
|(B\phi)(\theta,x)| & \leq   \left|\phi(\theta_1 x) - \phi(x)\right| + \sum_{i\geq 2}\left|
\phi(\theta_i x) - \phi(0)\right| \\
& \leq   C_{\phi} x^{\lambda}(1-\theta_1 )^{\lambda} + C_{\phi} x^{\lambda}\sum_{i\geq 2}
\theta_i^{\lambda}.
\tag*{\qed}
\end{align*}

We are going to work with a distance between solutions depending on $\lambda$. The distance
$d_{\lambda}$ (\ref{Intro:DistanceCoag}) involves the primitives of the solution of (\ref{forteq2}),
thus we recall \cite[Lemma 3.2]{Well-Pdnss}.
\begin{lemma}\label{lemma32} 
For $c\in \mathcal M^+$ and $x\in (0,\infty)$, we put
\begin{equation}\label{def_Fc}
F^c(x) := \int_0^{\infty} \mathds 1_{(x,\infty)}(y)\, c(dy), 
\end{equation}

\noindent If $c\in \mathcal M^+_{\lambda}$ for some $\lambda \in (0,1]$, then
\begin{equation*}
\int_0^{\infty} x^{\lambda-1}F^c(x)\,dx = M_{\lambda}(c)/\lambda,\hspace*{1cm} \lim_{x\rightarrow 0}
x^{\lambda} F^c(x) = \lim_{x\rightarrow \infty} x^{\lambda} F^c(x) =0,
\end{equation*}
and $F^c \in L^{\infty}(\varepsilon,\infty)$ for each $\varepsilon >0$.
\end{lemma}

We give now a very important inequality on which the existence and uniqueness proof relies.

\begin{proposition}\label{Prop_Uniqueness}
Consider $\lambda \in (0,1]$, a coagulation kernel $K$, 
a fragmentation kernel $F$ and a measure
$\beta$ on $\Theta$ satisfying Hypotheses
 $\ref{Def_Kernels} $ and $\ref{Hyp_beta}$ with the same
$\lambda$. Let $c^{in}$ and $d^{in} \in \mathcal M^+_{\lambda}$ and denote by
$(c_t)_{t\in[0,\infty)}$ a $(c^{in},K,F,\beta,\lambda)$-weak solution to 
$(\ref{weakeq2})$ and by
$(d_t)_{t\in[0,\infty)}$ a $(d^{in},K,F,\beta,\lambda)$-weak solution to 
$(\ref{weakeq2}).$ In
addition, we put $E(t,x) = F^{c_t}(x) - F^{d_t}(x)$, $\rho(x) = x^{\lambda-1}$ and
\begin{equation*}
R(t,x) = \int_0^{x} \rho(z) sign(E(t,z))\,dz \,\,\, \textrm{for}\,\,
(t,x)\in[0,\infty)\times(0,\infty).
\end{equation*}
  Then, for each $t\in [0,\infty)$, $R(t,\cdot) \in \mathcal H_{\lambda}$ and $($recall
$(\ref{def_Fc})$ and $(\ref{Intro:DistanceCoag}))$
\begin{align}
\label{Prop_ineq}
\dfrac{d}{dt}d_{\lambda}(c_t,d_t) & =   \dfrac{d}{dt}\int_0^{\infty} x^{\lambda-1} |E(t,x)|
dx\nonumber \\
& \leq   \frac{1}{2} \int_0^{\infty} \int_0^{\infty} K(x,y) \left[\rho(x+y) - \rho(x)
\right](c_t+d_t)(dy) |E(t,x)|\,dx\nonumber\\
&   +  \frac{1}{2} \int_0^{\infty} \int_0^{\infty} \partial_x K(x,y) \left(AR(t)\right)(x,y)
(c_t+d_t)(dy) E(t,x)\,dx \nonumber\\
& +  \int_0^{\infty} F'(x) \int_{\Theta} (BR(t))(\theta,x)\beta(d\theta) E(t,x) dx \nonumber \\
&  + \int_0^{\infty}F(x) x^{\lambda-1} |E(t, x)| \int_{\Theta} 
 \Big ( \sum_{i\geq 1}
\theta_i^{\lambda} - 1 \Big )\beta(d\theta)dx.
\end{align}
\end{proposition}

Note that it is straightforward that under the notation and assumptions of Proposition
\ref{Prop_Uniqueness}., from (\ref{Hyp_F1}), (\ref{Hyp_F2}), (\ref{Clambdabounds}) and using Lemma
\ref{Aux_ineq}. below, there exists a positive constant $C_1$ depending on $\lambda$, $\kappa_0$ and
$\kappa_1$ and a positive constant $C_2$ depending on $\kappa_2$, $\kappa_3$ and
$C^{\lambda}_{\beta}$ such that for each $t \in [0,\infty)$,
\begin{equation}\label{unicity}
\dfrac{d}{dt}d_{\lambda}(c_t,d_t) \leq \left(C_1 M_{\lambda}(c_t+d_t)+ C_2\right)
d_{\lambda}(c_t,d_t).
\end{equation}
Before to give the proof of Proposition \ref{Prop_Uniqueness}., we state two auxiliary results. In
Lemma \ref{Aux_ineq}. are given some inequalities which are useful to verify that the integrals on
the right-hand side of (\ref{Prop_ineq}) are convergent, and in Lemma \ref{lemma35}. we study the
time differentiability of $E$.

\begin{lemma}\label{Aux_ineq}
Under the notation and assumptions of Proposition $\ref{Prop_Uniqueness} ,$
 there exists a positive
constant $C$ such that for $(t,x,y) \in [0,\infty)\times (0,\infty)^2$,
\begin{eqnarray}
K(x,y)\left|\rho(x+y) - \rho(x)\right| & \leq & C x^{\lambda-1}y^{\lambda}, \nonumber \\
K(x,y)\left|\left(AR(t)\right)(x,y)\right| & \leq & C x^{\lambda}y^{\lambda}, \nonumber\\
\left|\partial_x K(x,y)\left(AR(t)\right)(x,y)\right| & \leq & C x^{\lambda-1}y^{\lambda},
\nonumber\\
\int_{\Theta}\left| (BR(t))(\theta,x)\right|\beta(d\theta) & \leq & C C^{\lambda}_{\beta}
x^{\lambda}.\label{AuxIf}
\end{eqnarray}
\end{lemma}

\noindent
{\bf Proof.}
\noindent The first three inequalities were proved in \cite[Lemma 3.4]{Well-Pdnss}. In particular,
recall that
\begin{equation}\label{R_inequality}
\left|\left(AR(t)\right)(x,y)\right|  \leq  \tfrac{2}{\lambda} (x\wedge y)^{\lambda}, 
\end{equation}
for $(t,x,y) \in [0,\infty)\times (0,\infty)^2$. Next, using (\ref{Clambdabounds}) we deduce 
\begin{align*}
&
\int_{\Theta}  | (BR(t))   (\theta,x) |\beta(d\theta)
   =   \Big |\int_{\Theta} \Big [
\sum_{i\geq 1} R(t,\theta_i x) - R(t,x) \Big ] \beta(d\theta) \Big | 
\\ & =   \int_{\Theta}
 \Big | \sum_{i\geq 2}\int_0^{\theta_i x}\partial_x R(t,z)dz - \int_{\theta_1 x}^{x}\partial_x R(t,z)
dz \Big |\beta(d\theta)\\
& \leq  \int_{\Theta}  \Big [ \sum_{i\geq 2}\int_0^{\theta_i x}z^{\lambda - 1} dz + \int_{\theta_1
x}^{x}z^{\lambda - 1} dz \Big ]\beta(d\theta) \leq  \frac{1}{\lambda} C^{\lambda}_{\beta}
x^\lambda.
\tag*{\qed}
\end{align*}

\begin{lemma}\label{lemma35}
Consider $\lambda \in (0,1]$, a coagulation kernel $K$, a fragmentation kernel $F$ and a measure
$\beta$ on $\Theta$ satisfying the Hypotheses $\ref{Def_Kernels}$
 with the same $\lambda$. Let
$c^{in} \in \mathcal M^+_{\lambda}$ and denote by $(c_t)_{t\in[0,\infty)}$ a
$(c^{in},K,F,\beta,\lambda)$-weak solution to (\ref{weakeq2}). Then
$(x,t)\mapsto \partial_t F^{c_t}(x)$ belongs to $L^{\infty}(0,s;L^1(0,\infty;x^{\lambda-1}dx))$, for
each $s \in [0,\infty)$.
\end{lemma}

\begin{proof}[\bf Proof.]
Following the same ideas as in \cite{Well-Pdnss}, we consider $\vartheta \in \mathcal C([0,\infty))$
with compact support in $(0,\infty)$, we put
\begin{equation*}
\phi(x) = \int_0^x \vartheta(y)\,dy,\hspace*{5mm} \textrm{for}\,\, x\in(0,\infty),
\end{equation*}
this function belongs to $\mathcal H_{\lambda}$. First, performing an integration by parts and using
Lemma \ref{lemma32}. we obtain
\begin{eqnarray*}
\int_0^{\infty} \vartheta(x) F^{c_t}(x)\,dx & = & \int_0^{\infty} \phi(x) c_t(dx).
\end{eqnarray*}
Next, on the one hand recall that in  \cite[eq. (3.7)]{Well-Pdnss} was proved that
\begin{align*}
&  \int_0^{\infty} \int_0^{\infty} K(x,y)\left(A\phi\right)(x,y)  \,c_t(dy)\,c_t(dx) dz \\
&\qquad  = \int_0^{\infty} \vartheta(z) \int_0^z \int_0^z \mathds 1_{[z,\infty)}(x+y)
K(x,y) c_t(dy)\,c_t(dx) dz \\
& \qquad - \int_0^{\infty} \vartheta(z) \int_z^{\infty} \int_z^{\infty} K(x,y)
c_t(dy)\,c_t(dx) dz.
\end{align*}
On the other hand, using the Fubini Theorem, we have
\begin{align*}
& \int_0^{\infty} F(x) \int_{\Theta} \left(B\phi\right)(\theta,x)\beta(d\theta)c_t(dx) \\ 
& \qquad  = 
\int_0^{\infty} F(x) \int_{\Theta} 
 \Big[ \sum_{i\geq 1} \int_0^{\theta_i x}
\vartheta(z) dz - \int_{0}^{x}\vartheta(z)dz  \Big] \beta(d\theta)\,c_t(dx)\\
&\qquad =  \int_0^{\infty} \vartheta(z) \int_{\Theta} 
 \Big [ \sum_{i\geq 1}
\int_{z/\theta_i}^{\infty} F(x) c_t(dx) - \int_{z}^{\infty} F(x) c_t(dx) 
\Big]\beta(d\theta)\,dz.
\end{align*}
Thus, from (\ref{weakeq2}) we infer that 
\begin{align*}
& 
\frac{d}{dt} \! \int_0^{\infty}  \!  \! \vartheta(x) F^{c_t}(x) dx  \!  =  \! 
  \frac{1}{2} \! \int_0^{\infty} \!  \! 
\vartheta(z) \int_0^z \int_0^z \mathds 1_{[z,\infty)}(x+y) K(x,y) c_t(dy)\,c_t(dx) dz \\
&\qquad  - \frac{1}{2}\int_0^{\infty} \vartheta(z) \int_z^{\infty} \int_z^{\infty} K(x,y)
c_t(dy)\,c_t(dx) dz \\
&\qquad   + \int_0^{\infty} \vartheta(z) \int_{\Theta} 
 \Big [ \sum_{i\geq 1} \int_{z/\theta_i}^{\infty}
F(x) c_t(dx) - \int_{z}^{\infty} F(x) c_t(dx)  \Big ]\beta(d\theta)\,dz,
\end{align*}
whence 
\begin{align}\label{weakD_F}
& 
\partial_t F^{c_t}(z)  =  \frac{1}{2}\int_0^z \int_0^z \mathds 1_{[z,\infty)}(x+y) K(x,y)
c_t(dy)\,c_t(dx) \nonumber \\
& - \frac{1}{2} \int_z^{\infty} \int_z^{\infty} K(x,y) c_t(dy)\,c_t(dx)\nonumber \\
&  + \int_{\Theta} \Big[ \sum_{i\geq 1} \int_{z/\theta_i}^{\infty} F(x) c_t(dx)\beta(d\theta)-
\int_{z}^{\infty} F(x) c_t(dx)\Big]\beta(d\theta),
\end{align}
for $(t,z)\in[0,\infty)\times (0,\infty)$. First, in \cite[Lemma 3.5]{Well-Pdnss} it was shown that,\begin{align*}
& \int_0^{\infty}z^{\lambda-1}\Big|\frac{1}{2}\int_0^z \int_0^z \mathds 1_{[z,\infty)}(x+y) K(x,y)
c_t(dy)\,c_t(dx) \\
& \qquad\qquad\qquad- \frac{1}{2} \int_z^{\infty} \int_z^{\infty} K(x,y) c_t(dy)\,c_t(dx) \Big|dz \, \leq \, \frac{2\kappa_0}{\lambda} M_{\lambda}(c_t)^2.
\end{align*}
Thus, from (\ref{Hyp_F1}) and the Fubini Theorem follows that, for each $t\in[0,\infty)$,
\begin{align*}
& \int_0^{\infty} z^{\lambda - 1} |\partial F^{c_t}(z)| dz  \leq  \frac{2\kappa_0}{\lambda}
M_{\lambda}(c_t)^2 \\
& +\int_0^{\infty}z^{\lambda - 1} \Big| \int_{\Theta}\Big(\sum_{i\geq 2}
\int_{z/\theta_i}^{\infty}F(x) c_t(dx) - \int_{z}^{z/\theta_1}F(x) c_t(dx)
\Big)\Big|\beta(d\theta)\,dz\\
&\leq  \frac{2\kappa_0}{\lambda} M_{\lambda}(c_t)^2 + \kappa_2 \int_{\Theta}\int_0^{\infty}\Big[
\sum_{i\geq 2}\int_0^{\theta_i x}z^{\lambda - 1} dz + \int_{\theta_1 x}^{x}z^{\lambda - 1}
\Big] c_t(dx) \beta(d\theta)\\
&\leq  \frac{2\kappa_0}{\lambda} M_{\lambda}(c_t)^2 + \frac{\kappa_2}{\lambda} M_{\lambda}(c_t)
\Big[ \int_{\Theta} \Big(\sum_{i\geq 2} \theta^{\lambda}_i + (1-\theta^{\lambda}_1)
\Big)\beta(d\theta) \Big]\\
&\leq  \frac{2\kappa_0}{\lambda} M_{\lambda}(c_t)^2 + \frac{C_{\beta}^{\lambda} \kappa_2}{\lambda}
M_{\lambda}(c_t),
\end{align*}
where we have used (\ref{Hyp2_Beta}). Finally, since the right-hand side of the above inequality is
bounded on $[0,t]$ for all $t>0$ by (\ref{Def_finitemoment}), we obtain the expected result.
\end{proof} 

\begin{proof}[Proof of Proposition \ref{Prop_Uniqueness}] 
Let $t\in [0,\infty)$. We first note that, since $s\mapsto M_{\lambda}(c_s)$ and $s\mapsto
M_{\lambda}(d_s)$ are in $L^{\infty}(0,t)$ by (\ref{Def_finitemoment}), it follows from Lemmas
\ref{lemma32}. and \ref{Aux_ineq}. that the integrals in (\ref{Prop_ineq}) are absolutely
convergent. Furthermore, for $t\geq0$ and $x>y$, we have
\begin{align*}
|R(t,x) - R(t,y)| &=  \Big|\int_y^x z^{\lambda-1} sign(E(t,z))\,dz \Big|\\
& \leq  \frac{1}{\lambda}(x^\lambda - y^\lambda) = \frac{1}{\lambda}\left((x-y+y)^\lambda -
y^\lambda\right) \\
&\leq  \frac{1}{\lambda} (x-y)^\lambda,
\end{align*}
since $\lambda \in(0,1]$. Thus $R(t,\cdot) \in \mathcal H_{\lambda}$ for each $t\in[0,\infty)$. 

\noindent Next, by Lemmas \ref{lemma32} and \ref{lemma35}, $E \in
W^{1,\infty}(0,s;L^1(0,\infty;x^{\lambda-1}dx))$ for every $s\in (0,T)$, so that
\begin{align*}
\dfrac{d}{dt}\int_0^{\infty} x^{\lambda-1} |E(t,x)| dx & =  \int_0^{\infty}x^{\lambda-1}
sign(E(t,x))\, \partial_t E(t,x) \,dx\\
& =  \int_0^{\infty} \partial_x R(t,x) \big( \partial_t F^{c_t}(x) - \partial_t F^{d_t}(x) \big)
dx.
\end{align*} 
We use (\ref{weakD_F}) to obtain
\begin{align}\label{DifferentialE}
&\dfrac{d}{dt}\int_0^{\infty} x^{\lambda-1} |E(t,x)| dx  \\
&= \frac{1}{2}\int_0^{\infty} \partial_x R(t,z) \int_0^z \int_0^z \mathds
1_{[z,\infty)}(x+y) K(x,y) (c_t(dy)\,c_t(dx)\nonumber \\
&\hspace{9cm}-d_t(dy)\,d_t(dx))dz\nonumber \\
&-\frac{1}{2}\int_0^{\infty} \partial_x R(t,z)\int_z^{\infty} \int_z^{\infty}
K(x,y)(c_t(dy)\,c_t(dx)-d_t(dy)\,d_t(dx))dz\nonumber\\
&+\int_0^{\infty} \partial_x R(t,z)\int_{\Theta}\Big[\sum_{i\geq 1}
\int_{z/\theta_i}^{\infty} F(x) (c_t-d_t)(dx)\\
&\hspace{6cm}-\int_{z}^{\infty} F(x)
(c_t-d_t)(dx)\Big]\beta(d\theta)dz.\nonumber
\end{align}

Recalling \cite[eq. (3.8)]{Well-Pdnss} and using the Fubini Theorem we obtain 
\begin{align}\label{EqIcIf}
& \dfrac{d}{dt}\int_0^{\infty} x^{\lambda-1} |E(t,x)| dx  \\
& \qquad =  \frac{1}{2} \int_0^{\infty} I^c(t,x)
\left(c_t - d_t\right)(dx) + \int_0^{\infty} I^f(t,x) \left(c_t - d_t\right)(dx),\nonumber
\end{align}
where
\begin{align*}
& I^c(t,x)  =  \int_0^{\infty}K(x,y) (AR(t))(x,y)(c_t + d_t)(dy),\hspace{1cm} x\in(0,\infty)\\
& I^f(t,x)  =  F(x) \int_{\Theta} (BR(t))(\theta,x) \beta(d\theta),\hspace{2.7cm} x\in(0,\infty).\\
\end{align*}

It follows from (\ref{AuxIf}) with (\ref{Hyp_F1}) that
\begin{align}\label{If_bouded}
|I^f(t,x)| \leq C\, x^{\lambda}, \,\,\,x\in(0,\infty),\,\,\,t\in[0,\infty).
\end{align}

We would like to be able to perform an integration by parts in the second integral of the right hand
of (\ref{EqIcIf}). However, $I^f$ is not necessarily differentiable with respect to $x$. We thus fix
$\varepsilon \in (0,1)$ and put
\begin{align*}
I_{\varepsilon}^f(t,x) = F(x)\int_{\Theta} (BR(t))(\theta,x) \beta_{\varepsilon}(d\theta),
\hspace{1cm} x\in(0,\infty),
\end{align*}
where $\beta_{\varepsilon}$ is the finite measure $\beta |_{\Theta_{\varepsilon}}$ with
$\Theta_{\varepsilon} = \{ \theta\in\Theta : \theta_1 \leq 1 -\varepsilon\}$ and note that
\begin{align}\label{Thetan_finite1}
\beta_{\varepsilon}(\Theta) = \int_{\Theta} \mathds 1_{\{1-\theta_1 \geq \varepsilon\}} \,
\beta(d\theta) \leq \frac{1}{\varepsilon} \int_{\Theta} (1-\theta_1) \, \beta(d\theta)
\leq\frac{1}{\varepsilon} \, C^{\lambda}_{\beta} < \infty.
\end{align}

\noindent Since $F$ belongs to $W^{1,\infty}(\alpha,1/\alpha)$ for $\alpha \in (0,1)$ and
$|R(t,x)|\leq x^{\lambda}/\lambda$ and $|\partial_x R(t,x)|\leq x^{\lambda-1}$ we deduce that
$I^f_{\varepsilon} \in W^{1,\infty}(\alpha,1/\alpha)$ for $\alpha \in (0,1)$ with
\begin{align}\label{Ifderive}
& \partial_x I^f_{\varepsilon}(t,x) = F'(x) \int_{\Theta}
(BR(t))(\theta,x)\beta_{\varepsilon}(d\theta) \\
&\qquad \qquad + F(x) \int_{\Theta} \Big[ \sum_{i\geq
1}\theta_i\partial_x R(t,\theta_i x)- \partial_x R(t,x)\Big]\beta_{\varepsilon}(d\theta).\nonumber
\end{align}

\noindent We now perform an integration by parts to obtain
\begin{align}
& \int_0^{\infty} I^f(t,x) ( c_t - d_t)(dx)  =  \int_0^{\infty} \big(I^f -
I_{\varepsilon}^f \big)(t,x)( c_t - d_t )(dx) \label{IfIPP}\\
&\qquad\qquad  - \big [I_{\varepsilon}^f(t,x)
E(t,x)\big ]_{x=0}^{x=\infty} + \int_0^{\infty} \partial_x I_{\varepsilon}^f(t,x) E(t,x)dx. \nonumber
\end{align} 

\noindent First, we have 
\begin{align*}
& \Big| \int_0^{\infty} \big(I^f - I_{\varepsilon}^f\big)(t,x)( c_t - d_t)(dx)
\Big| \leq \int_0^{\infty} \big|\big(I^f - I_{\varepsilon}^f\big)(t,x)\big|\left( c_t +
d_t\right)(dx) \\
& \leq\kappa_2 \int_0^{\infty} \int_{\Theta} \left|(BR(t))(\theta,x)\right|
(\beta-\beta_{\varepsilon})(d\theta)(c_t+d_t)(dx)\\
& \leq \kappa_2 \int_0^{\infty} \int_{\Theta} \Big [ \sum_{i\geq 2}\int_0^{\theta_i
x}z^{\lambda - 1} dz + \int_{\theta_1 x}^{x}z^{\lambda - 1} dz\Big] \mathds 1_{\{1-\theta_1 <
\varepsilon\}}\beta(d\theta)(c_t+d_t)(dx)\\
& \leq \frac{\kappa_2}{\lambda} \int_0^{\infty}x^\lambda\int_{\Theta}
\Big[\sum_{i\geq 2}\theta_i^{\lambda}+(1-\theta_1)^{\lambda}\Big] \mathds 1_{\{1-\theta_1 <
\varepsilon\}}\beta(d\theta)(c_t+d_t)(dx) \\
& = \frac{\kappa_2}{\lambda} M_{\lambda}(c_t+d_t) \int_{\Theta} \Big[\sum_{i\geq
2}\theta_i^{\lambda}+(1-\theta_1)^{\lambda}\Big] \mathds 1_{\{1-\theta_1 <
\varepsilon\}}\beta(d\theta),
\end{align*}
whence, recalling (\ref{Hyp2_Beta})  
\begin{align}\label{AuxIPP1}
\lim_{\varepsilon \rightarrow 0} \int_0^{\infty} \big(I^f - I_{\varepsilon}^f\big)(t,x)\left( c_t
- d_t\right)(dx) = 0.
\end{align}
Next, it follows from (\ref{If_bouded}) that
\begin{align*}
|I_{\varepsilon}^f(t,x) E(t,x)| \leq C x^{\lambda} \big(F^{c_t}(x) + F^{d_t}(x) \big),
\,\,\,x\in(0,\infty),\,\,\,t\in[0,\infty),
\end{align*}
we can thus easily conclude by Lemma \ref{lemma32}. that
\begin{align}\label{AuxIPP2}
\lim_{x\rightarrow 0} I_{\varepsilon}^f(t,x) E(t,x) = \lim_{x\rightarrow \infty}
I_{\varepsilon}^f(t,x) E(t,x) = 0.
\end{align}
Finally, (\ref{Hyp_F2}), Lemma \ref{lemma32}. and (\ref{AuxIf}) imply that 
\begin{align}\label{AuxIPP3}
& \lim_{\varepsilon \rightarrow 0} \int_0^{\infty} F'(x) \int_{\Theta}
(BR(t))(\theta,x)\beta_{\varepsilon}(d\theta)
E(t,x) dx   \\
&\qquad\qquad =\int_0^{\infty} F'(x) \int_{\Theta} (BR(t))(\theta,x)\beta(d\theta) E(t,x) dx, \nonumber
\end{align}
while 
\begin{align}
&\limsup_{\varepsilon \rightarrow 0} \int_0^{\infty} F(x) \int_{\Theta} \Big[ \sum_{i\geq
1}\theta_i\partial_x R(t,\theta_i x)- \partial_x R(t,x)\Big]\beta_{\varepsilon}(d\theta) E(t,x) dx
 \nonumber\\
&= \limsup_{\varepsilon \rightarrow 0}\int_0^{\infty} F(x)\int_{\Theta} \Big[ \sum_{i\geq 1}
\theta_i^{\lambda} x^{\lambda-1} sign(E(t,\theta_i x)) - x^{\lambda-1} sign(E(t,
x))\Big ] \nonumber \\
& \hspace{9cm} \times \beta_{\varepsilon}(d\theta)E(t,x) dx\nonumber\\
& = \limsup_{\varepsilon \rightarrow 0}\int_0^{\infty} F(x) x^{\lambda-1} sign(E(t, x))E(t, x) \nonumber\\
& \qquad\qquad \times \int_{\Theta} \Big [ \sum_{i\geq 1} \theta_i^{\lambda} sign\big(E(t,\theta_i x)E(t, x) \big) - 1
\Big ]\beta_{\varepsilon}(d\theta)dx \nonumber\\
& \leq \limsup_{\varepsilon \rightarrow 0}\int_0^{\infty}F(x) x^{\lambda-1} |E(t, x)| \int_{\Theta} \Big( \sum_{i\geq 1} \theta_i^{\lambda} - 1 \Big )\beta_{\varepsilon}(d\theta)dx\nonumber \\
& = \int_0^{\infty}F(x) x^{\lambda-1} |E(t, x)| \int_{\Theta} \Big[ \sum_{i\geq 1} \theta_i^{\lambda} - 1 \Big]\beta(d\theta)dx. \label{AuxIPP4}
\end{align}
We have used (\ref{Clambdabounds}) and (\ref{Hyp2_Beta}). Note that we are only interested in an
upper bound, when the term $ \sum_{i\geq 1} \theta_i^{\lambda} - 1$ is negative, $0$ would be a
better bound for the last term.

Recall (\ref{EqIcIf}), the term involving $I^c$ was treated in \cite[Proposition 3.3]{Well-Pdnss},
while from (\ref{IfIPP}) with (\ref{AuxIPP1}), (\ref{AuxIPP2}), (\ref{AuxIPP3}) and (\ref{AuxIPP4})
we deduce the inequality (\ref{Prop_ineq}), which completes the proof of Proposition
\ref{Prop_Uniqueness}.
\end{proof}

\subsection{Proof of Theorem \ref{Theorem}.} 
\begin{proof}[Uniqueness]
Owing to (\ref{Def_finitemoment}) and (\ref{unicity}), the uniqueness assertion of Theorem
\ref{Theorem}. readily follows from the Gronwall Lemma.
\end{proof}
\begin{proof}[Existence]
The proof of the existence assertion of Theorem \ref{Theorem}. is split into three steps. The first
step consists in finding an approximation to the coagulation-fragmentation equation by a version of
(\ref{weakeq2}) with finite operators: we will show existence in the set of positive measures with
finite total variation, i.e. $\mathcal M^+_0$, using the Picard method.

Next, we will show existence of a weak solution to (\ref{forteq2}) with an initial condition
$c^{in}$ in $\mathcal M_{\lambda}^+ \cap \mathcal M_2^+$, the final step consists in extending this
result to the case where $c^{in}$ belongs only to $\mathcal M_{\lambda}^+$.

\noindent \textbf{Bounded Case : existence and uniqueness in $\mathcal M^+_0$.-}  

We consider a bounded coagulation kernel and a fragmentation mechanism which gives only a finite
number of fragments. This is
\begin{equation}\label{Hypfinite}\left\lbrace
\begin{array}{rccr}
K(x,y) & \leq & \overline K,&\textrm{ for some } \overline K\in \mathbb R^+\\
F(x) & \leq & \overline F,&\textrm{ for some } \overline F\in \mathbb R^+\\
\beta(\Theta) & < & \infty, & \\
\beta(\Theta\setminus \Theta_k) & = & 0,&\textrm{ for some } k\in\mathbb N,
\end{array} \right.
\end{equation}
where 
\begin{equation*}
\Theta_k = \left\lbrace \theta = (\theta_n)_{n \geq 1} \in \Theta :\, \theta_{k+1} =
\theta_{k+2}=\ldots=0\right\rbrace.
\end{equation*}
We will show in this paragraph that under this assumptions there exists a global weak-solution to
(\ref{forteq2}). We will use the notation $\|\cdot\|_{\infty}$ for the $\sup$ norm on
$L^\infty[0,\infty)$ and $\|\cdot\|_{VT}$ for the total variation norm on measures. The result reads
as follows.

\begin{proposition}\label{existence_fini}
Consider $\mu^{in} \in\mathcal M^+_0$. Assume that the coagulation and fragmentation kernels $K$ and
$F$ and the measure $\beta$ satisfy the assumptions (\ref{Hypfinite}). Then, there exists a unique
non-negative weak-solution $(\mu_t)_{t\geq0}$ starting at $\mu_0 = \mu^{in}$ to (\ref{forteq2}).
Furthermore, it satisfies for all $t\geq 0$,
\begin{equation}\label{mu_bound1}
\sup_{[0,t]} \| \mu_s\|_{VT} \leq C_t \,\| \mu^{in}\|_{VT}, 
\end{equation} 
where $C_t$ is a positive constant depending on $t$, $\overline{K}$, $\overline{F}$ and $\beta$.
\end{proposition}

\begin{remark}\label{RemarkWeakSol}
Proposition \ref{existence_fini}. deals with weak solutions to (\ref{forteq2}) with
$\mu^{in}\in\mathcal M_0^+$ and with respect to the set of test functions $\phi \in
L^{\infty}([0,\infty))$. However, when $\mu^{in}\in\mathcal M_{\lambda}^+$, we can apply equation
(\ref{weakeq2}) with $\phi(x) = x^{\lambda}\wedge A$ with $A>0$, the Gronwall Lemma and then make
tend $A$ to infinity to prove that
\begin{equation*}
\sup_{[0,T]} M_{\lambda}(\mu_t) < \infty,\,\, \forall T\geq 0.
\end{equation*}

In the same way, using this last bound together with (\ref{Hypfinite}), (\ref{mu_bound1}) and the
Lebesgue dominated convergence Theorem, we extend readily to $\phi\in \mathcal H_{\lambda}$. Hence,
whenever $\mu^{in}\in\mathcal M_{\lambda}^+$ we obtain a $(\mu^{in},K,F,\beta,\lambda)$-weak
solution $(\mu_t)_{t\geq0}$ to (\ref{weakeq2}).
\end{remark}

\noindent To prove this proposition we need to replace the operator $A$ in (\ref{weakeq2}) by an
equivalent one, this new operator will be easier to manipulate. We consider, for $\phi$ a bounded
function, the following operators
\begin{align}
(\tilde A\phi)(x,y) & =  K(x,y) \Big [\frac{1}{2} \phi(x+y)-\phi(x) \Big ],\label{operatorA2}\\
(L\phi)(x)  & =   F(x) \int_{\Theta} \Big(\sum_{i\geq 1} \phi(\theta_i x) - \phi(x) \Big)
\beta(d\theta). \label{operatorB2}
\end{align}

Thus, (\ref{weakeq2}) can be rewritten as
\begin{align}
&\frac{d}{dt} \int_0^ {\infty} \phi(x) c_t(dx) 
 =  \int_0^{\infty}\Big[\int_0^{\infty}(\tilde A \phi)(x,y) c_t(dy) + (L\phi)(x)\Big] c_t(dx).\nonumber \\
& \label{weakv2}
\end{align}

The Proposition will be proved using an implicit scheme for equation (\ref{weakv2}). First, we need
to provide a unique and non-negative solution to this scheme.

\begin{lemma}\label{positivesol}
Consider $\mu^{in} \in \mathcal M^+_0$ and let $(\nu_t)_{t\geq 0}$ be a family of measures in
$\mathcal M^+_0$ such that $\sup_{[0,t]} \| \nu_s\|_{VT} < \infty$ for all $t\geq 0$. Then, under
the assumptions (\ref{Hypfinite}), there exists a unique non-negative solution $(\mu_t)_{t\geq0}$
starting at $\mu_0 = \mu^{in}$ to
\begin{align}
&\int_0^ {\infty} \phi(x) \mu_t(dx) 
\label{possoleq} \\
& =  \int_0^{\infty}\phi(x) \mu_0(dx) +\int_0^{t} \int_0^{\infty}\Big[\int_0^{\infty}(\tilde A
\phi)(x,y) \nu_s(dy) + (L\phi)(x)\Big] \mu_s(dx)ds \nonumber
\end{align}
for all $\phi \in L^{\infty}(\mathbb R^+)$. Furthermore, the solution satisfies for all $t\geq 0$,
\begin{equation}\label{mu_bound}
\sup_{[0,t]} \| \mu_s\|_{VT} \leq C_t\,\| \mu^{in}\|_{VT},
\end{equation} 
where $C_t$ is a positive constant depending on $t$, $\overline{K}$, $\overline{F}$ and $\beta$. The constant $C_t$ does not depend on  $\sup_{[0,t]} \| \nu_s\|_{VT}$.
\end{lemma}

We will prove this lemma in two steps. First, we show that (\ref{possoleq}) is equivalent to another
equation. This new equation is constructed in such a way that the negative terms of equation
(\ref{possoleq}) are eliminated. Next, we prove existence and uniqueness for this new equation. This
solution will be proved to be non-negative and it will imply existence, uniqueness and
non-negativity of a solution to (\ref{possoleq}).

\begin{proof}
\textbf{Step 1.-} First, we give an auxiliary result which allows to differentiate equation
(\ref{possoleq2}) when the test function depends on $t$.
\begin{lemma}\label{eqderivee}
Let $(t,x) \mapsto \phi_t(x) : \mathbb R^+ \times \mathbb R^+ \rightarrow \mathbb R$ be a bounded
measurable function, having a bounded partial derivative $\partial \phi /\partial t$ and consider
$(\mu_t)_{t\geq 0}$ a weak-solution to (\ref{possoleq}). Then, for all $t\geq 0$,
\begin{align*}
&\frac{d}{dt} \int_0^ {\infty} \phi_t(x) \mu_t(dx) = \int_0^ {\infty} \frac{\partial}{\partial
t}\phi_t(x) \mu_t(dx) \\
&+ \int_0^{\infty}\int_0^{\infty}(\tilde A \phi_t)(x,y) \mu_t(dx) \nu_t(dy) +
\int_0^{\infty} (L\phi_t)(x) \mu_t(dx).
\end{align*}
\end{lemma}
\begin{proof} First, note that for $0\leq t_1 \leq t_2$ we have,
\begin{align*}
& \int_0^ {\infty} \phi_{t_2}(x) \mu_{t_2}(dx) -  \int_0^ {\infty} \phi_{t_1}(x) \mu_{t_1}(dx) \\
& =\,\, \int_0^ {\infty} \left( \phi_{t_2}(x) - \phi_{t_1}(x) \right) \mu_{t_2}(dx)
+ \int_0^ {\infty} \phi_{t_1}(x)\left(\mu_{t_2} - \mu_{t_1} \right)(dx)\\
&= \,\, \int_{t_1}^{t_2} \int_0^ {\infty} \frac{\partial}{\partial t} \phi_s(x)
\mu_{t_2}(dx)ds + \int_{t_1}^{t_2} \frac{d}{dt} \int_0^ {\infty} \phi_{t_1}(x) \mu_t(dx) dt \\
& = \,\, \int_{t_1}^{t_2} \int_0^ {\infty} \frac{\partial}{\partial t} \phi_s(x)
\mu_{t_2}(dx) ds\\
& + \int_{t_1}^{t_2} \Big[ \int_0^{\infty}\int_0^{\infty}(\tilde A \phi_{t_1})(x,y)
\mu_s(dx) \nu_s(dy) + \int_0^{\infty} (L\phi_{t_1})(x) \mu_s(dx)\Big]ds.
\end{align*}
Thus, fix $t>0$ and set for $n\in\mathbb N$, $t_k = t\dfrac{k}{n}$ with $k=0,1,\ldots,n$, we get
\begin{align*}
&\int_0^ {\infty} \phi_t(x) \mu_t(dx)  =  \int_0^{\infty} \phi_0(x)\mu_0(dx) \\
& \qquad  + \sum_{k=1}^n \Big[
\int_0^ {\infty} \phi_{t_{k}}(x) \mu_{t_{k}}(dx) - \int_0^ {\infty} \phi_{t_{k-1}}(x)
\mu_{t_{k-1}}(dx)\Big] \\
& = \int_0^{\infty} \phi_0(x)\mu_0(dx) + \sum_{k=1}^n \int_{t_{k-1}}^{t_k} \int_0^ {\infty}
\frac{\partial}{\partial t} \phi_s(x) \mu_{t_k}(dx) ds+\sum_{k=1}^n \int_{t_{k-1}}^{t_k}\\
&\qquad \Big[ \int_0^{\infty}\int_0^{\infty}(\tilde A
\phi_{t_{k-1}})(x,y) \mu_s(dx) \nu_s(dy) + \int_0^{\infty} (L\phi_{t_{k-1}})(x) \mu_s(dx)\Big]ds.
\end{align*}
Next, for $s\in[t_{k-1},t_k)$ we set $k=\left\lfloor \frac{ns}{t}\right\rfloor$ and use the
notation $\overline s_n := t_k = \frac{t}{n} \left\lfloor \frac{ns}{t}\right\rfloor$ and
$\underline s_n := t_{k-1}$. Thus, the equation above can be rewritten as
\begin{align*}
& \int_0^ {\infty} \phi_t(x) \mu_t(dx) =  \int_0^{\infty} \phi_0(x)\mu_0(dx) + \int_0^t
\int_0^{\infty} \frac{\partial}{\partial t}\phi_s(x) \mu_{\overline s_n}(dx) ds \\
& + \int_0^t\int_0^{\infty}\int_0^{\infty}(\tilde A \phi_{\underline s_n})(x,y) \mu_{s}(dx)
\nu_s(dy) ds + \int_0^t \int_0^{\infty} (L\phi_{\underline s_n})(x) \mu_s(dx) ds,
\end{align*}
and the lemma follows from letting $n \rightarrow \infty$ since $\overline s_n \rightarrow s$.
\end{proof}

Next, we introduce a new equation. We put for $t\geq0$, 
\begin{equation}\label{gammaf}
\gamma_t(x) = \exp \Big[ \int_0^t \Big( \int_0^{\infty} K(x,y) \nu_s(dy) - F(x) \Big) ds \Big],\end{equation}
and we consider the equation
\begin{align}
& \frac{d}{dt} \int_0^ {\infty} \phi(x) \tilde \mu_t(dx)  =  \int_0^{\infty} \Big[
\int_0^{\infty} \frac{1}{2} K(x,y) (\phi\gamma_t)(x+y) \nu_t(dy)\nonumber \\
& + F(x) \int_{\Theta} \sum_{i\geq 1} (\phi\gamma_t) (\theta_i x) \beta(d\theta)
\Big]\gamma_t^{-1}(x) \tilde \mu_t(dx).\label{possoleq2}
\end{align}
Now, we give a result that relates (\ref{possoleq}) to (\ref{possoleq2}).
\begin{lemma}\label{equivsols}
Consider $\mu^{in} \in \mathcal M^+_0$ and recall (\ref{gammaf}). Then, $(\mu_t)_{t\geq 0}$ with
$\mu_0 = \mu^{in}$ is a weak-solution to (\ref{possoleq}) if and only if $(\tilde \mu_t)_{t\geq 0}$
with $\tilde \mu_0 = \mu^{in}$ is a weak-solution to (\ref{possoleq2}), where $\tilde \mu_t =
\gamma_t\mu_t$ for all $t\geq 0$.
\end{lemma}
\begin{proof}
First, assume that  $(\mu_t)_{t\geq 0}$ is a weak-solution to (\ref{possoleq}). 

We have $\dfrac{\partial}{\partial t}\gamma_t(x) = \gamma_t(x)\Big[ \displaystyle \int_0^{\infty}
K(x,y) \nu_t(dy) - F(x) \Big]$. Note that $\gamma_t$, $\gamma_t^{-1}$ and
$\frac{\partial}{\partial t}\gamma_t$ are bounded on $[0,t]$ for all $t \geq 0$, by (\ref{Hypfinite})
and since $\underset{[0,t]}\sup \| \nu_s\|_{VT} < \infty$.

Set $\tilde \mu_t = \gamma_t \mu_t$, recall (\ref{operatorA2}) and (\ref{operatorB2}), by Lemma
\ref{eqderivee}., for all bounded measurable functions $\phi$, we have
\begin{align*}
&\frac{d}{dt} \int_0^ {\infty} \phi(x) \tilde \mu_t(dx) = \int_0^ {\infty} \phi(x) \gamma_t(x)
\Big[\int_0^{\infty} K(x,y) \nu_t(dy) - F(x) \Big] \mu_t(dx)\nonumber \\
& \qquad +\int_0^{\infty} \int_0^{\infty}
\Big[\frac{1}{2}(\phi\gamma_t)(x+y)-(\phi\gamma_t)(x)\Big]K(x,y) \nu_t(dy)\mu_t(dx)\nonumber\\
&\qquad + \int_0^{\infty} F(x) \int_{\Theta} \Big(\sum_{i\geq 1} (\phi\gamma_t)(\theta_i
x) - (\phi\gamma_t)(x) \Big) \beta(d\theta)\mu_t(dx)\nonumber\\
&= \int_0^{\infty} \int_0^{\infty} \frac{1}{2} K(x,y) (\phi\gamma_t)(x+y)
\nu_t(dy) \mu_t(dx)\nonumber \\
&\qquad + \int_0^{\infty} F(x) \int_{\Theta} \sum_{i\geq 1} (\phi\gamma_t) (\theta_i x)
\beta(d\theta) \mu_t(dx)\nonumber\\
&= \int_0^{\infty} \Big[ \int_0^{\infty} \frac{1}{2} K(x,y) (\phi\gamma_t)(x+y)
\nu_t(dy) + F(x) \int_{\Theta} \sum_{i\geq 1} (\phi\gamma_t) (\theta_i x) \beta(d\theta)
\Big]\nonumber \\
& \hspace{9cm}\times \gamma_t^{-1}(x) \tilde \mu_t(dx),
\end{align*}
and the result follows.

For the reciprocal assertion, we assume that $(\tilde \mu_t)_{t\geq 0}$ is a weak-solution to
(\ref{possoleq2}), set $\mu_t = \gamma_t^{-1} \tilde \mu_t$ and we show in the same way that
$(\mu_t)_{t\geq 0}$ is a weak-solution to (\ref{possoleq}).
\end{proof}

We note that, since all the terms between the brackets are non-negative, the right-hand side of
equation (\ref{possoleq2}) is non-negative whenever $\tilde \mu_t \geq 0$. Thus, $\gamma_t$ is an
integrating factor that removes the negative terms of equation (\ref{possoleq}).

\noindent \textbf{Step 2.-} We define the following explicit scheme for (\ref{possoleq2}): we set
$\tilde \mu^0_t = \mu^{in}$ for all $t\geq 0$ and for $n \geq 0$
\begin{equation}\label{scheme2}\left\lbrace
  \begin{array}{lcl}
\dfrac{d}{dt}\displaystyle\int_0^ {\infty} \phi(x)\tilde \mu^{n+1}_t(dx) & = & \displaystyle
\int_0^{\infty} \Big[ \int_0^{\infty} \frac{1}{2} K(x,y) (\phi\gamma_t)(x+y) \nu_t(dy) \\[4mm]
&& + F(x) \displaystyle \int_{\Theta} \sum_{i\geq 1} (\phi\gamma_t) (\theta_i x)
\beta(d\theta) \Big]\gamma_t^{-1}(x) \tilde \mu_t^n(dx)\\
\tilde \mu^{n+1}_0 & = & \mu^{in}.
\end{array} \right.
\end{equation}
Recall (\ref{Hypfinite}), note that the following operators are bounded:
\begin{eqnarray}
\Big\|\gamma_t^{-1}(\cdot)\int_0^{\infty} \frac{1}{2} K(\,\cdot\,,y) (\phi\gamma_t)(\,\cdot+y)
\nu_t(dy)\Big\|_{\infty} & \leq & C_t \|\phi\|_{\infty},\label{Aux1bounded}\\
\Big\|\gamma_t^{-1}(\cdot)F(\cdot) \displaystyle \int_{\Theta} \sum_{i\geq 1} (\phi\gamma_t)
(\theta_i\, \cdot\,) \beta(d\theta) \Big\|_{\infty} & \leq & C_t
\|\phi\|_{\infty},\label{Aux2bounded}
\end{eqnarray}
where $C_t$ is a positive constant depending on $\overline K$, $\overline F$, $\beta$ and
$\sup_{[0,t]} \| \nu_s\|_{VT}$.

Thus, we consider $\phi$ bounded, integrate in time (\ref{scheme2}), use (\ref{Aux1bounded}) and
(\ref{Aux2bounded}) to obtain
\begin{eqnarray*}
\int_0^{\infty} \phi(x) \left(\tilde \mu^{n+1}_t(dx)-\tilde\mu^{n}_t(dx)\right) & \leq
&C_{1,t}\|\phi\|_{\infty} \int_0^t \left\| \tilde\mu^{n}_s-\tilde\mu^{n-1}_s\right\|_{VT} ds \\
& & + C_{2,t}\|\phi\|_{\infty} \int_0^t \left\| \tilde\mu^{n}_s-\tilde \mu^{n-1}_s\right\|_{VT} ds,
\end{eqnarray*}
note that the the difference of the initial conditions vanishes since they are the same. We take the
$\sup$ over $ \|\phi\|_{\infty} \leq 1$ and use $\sup_{[0,t]} \| \nu_s\|_{VT} < \infty$ to deduce
\begin{equation*}
\left\| \tilde \mu^{n+1}_t-\tilde\mu^{n}_t\right\|_{VT} \leq C_{t} \int_0^{t} \left\|\tilde
\mu^{n}_s-\tilde\mu^{n-1}_s\right\|_{VT}\,ds,
\end{equation*}
where $C_t$ is a positive constant depending on $\overline K$, $\overline F$, $\beta$, $\sup_{[0,t]}
\| \nu_s\|_{VT}$ and $\|\phi\|_{\infty} $.
Hence, by classical arguments, $(\tilde \mu_t^n)_{t\geq0}$ converges in $\mathcal M_0^+$ uniformly
in time to $(\tilde \mu_t)_{t\geq0}$ solution to (\ref{possoleq2}), and since $\tilde \mu^ n_t \geq
0$ for all $n$, we deduce $\tilde \mu_t \geq 0$ for all $t\geq 0$. The uniqueness for
(\ref{possoleq2}) follows from similar computations.

Thus, by Lemma \ref{equivsols}. we deduce existence and uniqueness of $(\mu_t)_{t\geq0}$ solution to
(\ref{possoleq}), and since $\tilde \mu_t \geq 0$ we have $\mu_t \geq 0$ for all $t\geq0$.

Finally, it remains to prove (\ref{mu_bound}). For this, we apply (\ref{possoleq}) with $\phi(x)
\equiv 1$, remark that $(\tilde A 1)(x,y) \leq 0$ and that $(L1)(x)\leq \overline F (k-1)
\beta(\Theta)$. Since $\mu_t\geq 0$ for all $t\geq 0$, this implies
\begin{equation*}
\left\| \mu_t\right\|_{VT} = \int_0^{\infty} \mu_t(dx) \leq \left\| \mu_0\right\|_{VT} + \overline F
(k-1) \beta(d\Theta)\int_0^t\left\| \mu_s\right\|_{VT} ds.
\end{equation*}
Using the Gronwall Lemma, we conclude
\begin{equation*}
\sup_{[0,t]}\left\| \mu_s\right\|_{VT} \leq \left\| \mu^{in}\right\|_{VT} e^{C t}\,\,\,\, \textrm{
for all } \,\, t\geq 0,
\end{equation*}
where $C$ is a positive constant depending only on $\overline{K}$, $\overline{F}$ and $\beta$. We
point out that the term $\sup_{[0,t]}\|\nu_s\|_{VT}$ is not involved since it is relied to the
coagulation part of the equation, which is negative and bounded by $0$. This ends the proof of Lemma
\ref{positivesol}.
\end{proof}
 
\begin{proof}[Proof of Proposition \ref{existence_fini}] We define the following implicit scheme for
(\ref{weakv2}): $\mu^0_t = \mu^{in}$ for all $t\geq 0$ and for $n\geq 0$,
\begin{equation}\label{scheme1}\left\lbrace
  \begin{array}{lcl}
\dfrac{d}{dt}\displaystyle\int_0^ {\infty} \phi(x) \mu^{n+1}_t(dx) & = &
\displaystyle\int_0^{\infty}\displaystyle\int_0^{\infty}(\tilde A \phi)(x,y) \mu^{n+1}_t(dx)
\mu^{n}_t(dy) \\[4mm]
&&+ \displaystyle\int_0^{\infty} (L\phi)(x) \, \mu^{n+1}_t(dx)\\
\mu^{n+1}_0 & = & \mu^{in}.
\end{array} \right.
\end{equation}
First, from Lemma \ref{positivesol}. for $n\geq0$ we have existence of $(\mu^{n+1}_t)_{t\geq 0}$
unique and non-negative solution to (\ref{scheme1}) whenever $(\mu^{n}_t)_{t\geq 0}$ is non-negative
and $\sup_{[0,t]} \| \mu^{n}_s\|_{VT} <\infty $ for all $t\geq 0$. Hence, since $\mu^{in} \in
\mathcal M_0^+$, by recurrence we deduce existence, uniqueness and non-negativity of
$(\mu^{n+1}_t)_{t\geq 0}$ for all $n\geq0$ solution to (\ref{scheme1}).

Moreover, from (\ref{mu_bound}), this solution is bounded uniformly in $n$ on $[0,t]$ for all $t\geq
0$ since this bound does not depend on $\mu^n_t$, i.e.,
\begin{equation}\label{solBunif}
\sup_{n\geq 1}\sup_{[0,t]} \| \mu^{n+1}_s\|_{VT} \leq C_t\, \| \mu^{in}\|_{VT}. 
\end{equation}

Next, note that the operators $\tilde A$ and $L$ are bounded:
\begin{eqnarray}
\|L\phi\|_{\infty} & \leq & \overline F (k+1) \beta(\Theta) \|\phi\|_{\infty},\label{Lbounded} \\
\Big\|\int_0^{\infty}(\tilde A\phi)(\,\cdot\,,y) \mu(dy)\Big\|_{\infty} & \leq &
\frac{3}{2}\overline K \|\phi\|_{\infty} \,\|\mu\|_{VT}.\label{Abounded}
\end{eqnarray}
\noindent From (\ref{Abounded}) and (\ref{Lbounded}),
\begin{align*}
& \frac{d}{dt} \int_0^{\infty} \phi(x) \left(\mu^{n+1}_t(dx)-\mu^{n}_t(dx)\right) \\
&  = \int_0^{\infty} \int_0^{\infty}(\tilde A \phi)(x,y) \left(
\mu^{n+1}_t(dx)\mu^{n}_t(dy)- \mu^{n}_t(dx)\mu^{n-1}_t(dy)\right)\\
& \qquad +\int_0^{\infty} (L\phi)(x) \, \left(\mu^{n+1}_t-\mu^{n}_t\right)(dx)\\
&  =\int_0^{\infty} \int_0^{\infty}(\tilde A \phi)(x,y) \left[ \left(
\mu^{n+1}_t-\mu^{n}_t\right)(dx) \mu_t^n(dy) + \mu_t^n(dx)\left(
\mu^{n}_t-\mu^{n-1}_t\right)(dy)\right]\\
&\qquad+\int_0^{\infty} (L\phi)(x) \, \left(\mu^{n+1}_t-\mu^{n}_t\right)(dx)\\
& \leq\frac{3}{2} \overline K \|\phi\|_{\infty} \left\| \mu^{n}_t\right\|_{VT}
\Big[ \int_0^{\infty} \left| \mu^{n+1}_t-\mu^{n}_t\right|(dx) + \int_0^{\infty} \left|
\mu^{n}_t-\mu^{n-1}_t\right|(dy)\Big]\\
&\qquad + \overline F (k+1) \beta(\Theta) \|\phi\|_{\infty} \left\|
\mu^{n+1}_t-\mu^{n}_t\right\|_{VT},
\end{align*}
implying,
\begin{align*}
&\frac{d}{dt} \int_0^ {\infty} \phi(x) \left(\mu^{n+1}_t(dx)-\mu^{n}_t(dx)\right)  \leq 
\|\phi\|_{\infty} \big( \frac{3}{2} \overline K \left\| \mu^{n}_t\right\|_{VT} + \overline F (k+1)
\beta(\Theta) \big)\\
&\qquad \times\left\| \mu^{n+1}_t-\mu^{n}_t\right\|_{VT}+ \frac{3}{2} \overline K \, \|\phi\|_{\infty} \, \left\|\mu^{n}_t\right\|_{VT}\,\left\|
\mu^{n}_t-\mu^{n-1}_t\right\|_{VT}.
\end{align*}
We integrate on $t$, take the $\sup$ over $ \|\phi\|_{\infty} \leq 1$, and use (\ref{solBunif}), to
deduce that there exist two constants $C_{1,t}$ and $C_{2,t}$ depending on $t$ but not on $n$ such
that
\begin{align*}
 \left\| \mu^{n+1}_t-\mu^{n}_t\right\|_{VT} 
 \leq  C_{1,t} \int_0^{t} \left\| \mu^{n+1}_s-\mu^{n}_s\right\|_{VT}\,ds + C_{2,t} \int_0^{t}
\left\| \mu^{n}_s-\mu^{n-1}_s\right\|_{VT}\,ds.
\end{align*}
Note that the difference of initial conditions vanishes since they are the same. We obtain using the
Gronwall Lemma.
\begin{equation*}
\left\| \mu^{n+1}_t-\mu^{n}_t\right\|_{VT} \leq C_{2,t}\, e^{t\, C_{1,t}} \int_0^{t} \left\|
\mu^{n}_s-\mu^{n-1}_s\right\|_{VT}\,ds.
\end{equation*}
Hence, by usual arguments, $(\mu_t^n)_{t\geq0}$ converges in $\mathcal M_0^+$ uniformly in time to
the desired solution, which is also unique. Moreover, for some finite constant $C$ depending on $t$,
$\overline{K}$, $\overline{F}$ and $\beta$, this solution satisfies (\ref{mu_bound1}) by
(\ref{solBunif}).

  This concludes the proof of Proposition \ref{existence_fini}.
\end{proof}

\noindent \textbf{Existence and uniqueness for $c^{in} \in \mathcal M_{\lambda}^+ \cap \mathcal
M_{2}^+$.-}

\noindent We are no longer under (\ref{Hypfinite}), more generally we assume Hypotheses
\ref{Def_Kernels}. and \ref{Hyp_beta}. This paragraph is devoted to show existence in the case where
the initial condition satisfies:
\begin{equation*}
c^{in} \in \mathcal M_{\lambda}^+ \cap \mathcal M_{2}^+.
\end{equation*}

\begin{proof} 
First, for $n\geq 1$, we consider $c^{in,n}(dx) = \mathds 1_{[1/n,n]} c^{in}(dx)$, this
measure belongs to $\mathcal M^+_0$ and satisfies
\begin{equation}\label{cinbounded}
\sup_{n\geq 1}M_{\lambda}(c^{in,n}) \leq M_{\lambda}(c^{in}).
\end{equation}  

We also note that $\big(F^{c^{in,n}}\big)$ converges towards $F^{c^{in}}$ in
$L^{1}(0,\infty;x^{\lambda-1}\,dx)$ as $n\rightarrow \infty$. Define $K_n$ by $K_n(x,y) = K(x,y)
\wedge n$ for $(x,y) \in (0,\infty)^2$. Notice that (\ref{Hyp_K1}) and (\ref{Hyp_K2}) warrant that
\begin{equation}\label{Hyp_Kn}
\begin{array}{rcl}
K_n(x,y) & \leq & \kappa_0(x+y)^{\lambda},\\
(x^{\lambda}\wedge y^{\lambda}) |\partial_x K_n(x,y)| & \leq & \kappa_1 x^{\lambda-1} y^{\lambda}.
\end{array}
\end{equation}
Furthermore, we consider the set $\Theta(n)$ defined by $$\Theta(n) = \Big\lbrace \theta \in \Theta
: \theta_1 \leq 1-\frac{1}{n} \Big\rbrace,$$ we consider also the projector
\begin{equation}\label{projector}
\begin{array}{llcl}
\psi_n: & \Theta & \rightarrow & \Theta_n \\
		& \theta & \mapsto & \psi_n(\theta) = (\theta_1,\ldots,\theta_n,0,\ldots),
\end{array}
\end{equation}
and we put
\begin{equation} \label{beta_n}
\beta_n  = \mathds 1_{\theta\in\Theta(n)}\beta \circ \psi_n^{-1}.
\end{equation}
The measure $\beta_n$ can be seen as the restriction of $\beta$ to the projection of $\Theta(n)$
onto $\Theta_n$. Note that $\Theta(n) \subset \Theta(n+1)$ and that since we have excluded the
degenerated cases $\theta_1=1$ we have $\bigcup_n \Theta(n) = \Theta$.

Then, $K_n$, $F$ and $\beta_n$ satisfy (\ref{Hypfinite}) (use (\ref{Thetan_finite1})) and since
$c^{in,n} \in \mathcal M_{0}^+$, we have from Proposition \ref{existence_fini}. (recall Remark
\ref{RemarkWeakSol}.) that for each $n\geq 1$, there exists a
$(c^{in,n},K_n,F,\beta_n,\lambda)$-weak solution $(c^n_t)_{t\geq0}$ to (\ref{weakeq2}). 

Note that since we have fragmentation it is not evident that $M_{\lambda}(c_t)$ remains finite in
time. We need to control $M_{\lambda}(c_t)$ to verify (\ref{Def_finitemoment}). For this, we set
$\phi(x) = x^{\lambda}$, from (\ref{weakeq2}) and since $(A\phi)(x,y) \leq 0$ we have
\begin{eqnarray*}
\frac{d}{dt}\int_0^{\infty} x^{\lambda} c^n_t(dx) & = & \frac{1}{2} \int_0^{\infty}
\int_0^{\infty}K_n(x,y) (A\phi)(x,y) c^n_t(dx)\,c^n_t(dy)\nonumber\\
& & + \int_{\Theta} \int_0^{\infty} F(x) \Big(\sum_{i\geq 1} \theta_i^{\lambda} - 1\Big)
x^{\lambda} c^n_t(dx)\beta_n(d\theta)\nonumber\\
&\leq & \kappa_2 \, C_{\beta}^{\lambda} M_{\lambda}(c^n_t),
\end{eqnarray*}
where we used that clearly $C^{\lambda}_{\beta_n} \leq C^{\lambda}_{\beta}$ for all $n\geq 1$
(recall (\ref{Hyp2_Beta})). Note also that if $\sum_{i\geq 1} \theta_i^{\lambda} - 1<0$ then
$M_{\lambda}(c^n_t) < M_{\lambda}(c_0)$.

Using the Gronwall Lemma and (\ref{cinbounded}) we deduce, for all $t\geq 0$
\begin{equation}\label{momentcnbounded}
\sup_{n\geq 1} \sup_{[0,t]}M_{\lambda}(c^n_s) \leq C_t,
\end{equation}
where $C_t$ is a positive constant. Next, apply (\ref{weakeq2}) with $\phi(x) = x^{2}$ and since
$\sum_{i\geq 1}\theta_i^2 - 1\leq 0$ the fragmentation part is negative. In \cite[Lemma
A.3.\,\textit{(ii)}]{Cepeda_Fournier} was shown that there exists a constant $C$ depending only on
$\lambda$ and $\kappa_0$ such that $K_n(x,y)|(A\phi)(x,y)|\leq K(x,y)|(A\phi)(x,y)|\leq C
(x^2y^{\lambda} + x^{\lambda}y^2)$. Thus,
\begin{eqnarray*}
\frac{d}{dt}\int_0^{\infty} x^{2} c^n_t(dx) & \leq & \frac{C}{2} \int_0^{\infty} \int_0^{\infty}
(x^2y^{\lambda} + x^{\lambda}y^2) \, c^n_t(dx)\,c^n_t(dy)\nonumber\\
&= & C M_{\lambda}(c^n_t)M_{2}(c^n_t).
\end{eqnarray*}
Using the Gronwall Lemma, we obtain
\begin{equation*}
M_{2}(c^n_t) \leq M_{2}(c^{in})\, e^{C \int_0^t M_{\lambda}(c^{n}_s)ds},
\end{equation*}
for $t\geq 0$ and for each $n\geq1$. We point out that $x^2 \notin \mathcal H_{\lambda}$, but we can
proceed as in Remark \ref{RemarkWeakSol}, considering $\phi(x) = x^2 \wedge A$ with $A>0$ and making
$A$ tend to infinity.

  Hence, using (\ref{momentcnbounded}) we get
\begin{equation}\label{moment2bounded}
\sup_{n\geq 1} \sup_{[0,t]}M_{2}(c^n_s) \leq C_t,
\end{equation}
where $C_t$ is a positive constant. 
\begin{sloppypar}
We set $E_n(t,x) = F^{c^{n+1}_t}(x) - F^{c^{n}_t}(x)$ and define $R_n(t,x)=\int_0^x z^{\lambda-1}
sign(E_n(t,x)) dz$.  Recall (\ref{weakD_F}) and (\ref{DifferentialE}),
\end{sloppypar}
\begin{align*}
&\dfrac{d}{dt}\int_0^{\infty} x^{\lambda-1} |E_n(t,x)| dx   \\
&= \frac{1}{2}\int_0^{\infty} \partial_x R_n(t,z) \int_0^z \int_0^z \mathds
1_{[z,\infty)}(x+y) K_{n+1}(x,y)\\
& \hspace{5cm} \times  (c^{n+1}_t(dy)\,c^{n+1}_t(dx)-c^n_t(dy)\,c^n_t(dx))dz \\
&-\frac{1}{2}\int_0^{\infty} \partial_x R_n(t,z)\int_z^{\infty} \int_z^{\infty}
K_{n+1}(x,y)\\
& \hspace{5cm} \times (c^{n+1}_t(dy)\,c^{n+1}_t(dx)-c^n_t(dy)\,c^n_t(dx)) dz\\
&+\int_0^{\infty} \partial_x R_n(t,z)\int_{\Theta}\sum_{i\geq 1}
\int_{z/\theta_i}^{\infty} F(x) (c^{n+1}_t-c^n_t)(dx)\beta_{n+1}(d\theta) dz \\
&- \int_0^{\infty} \partial_x R_n(t,z) \int_{\Theta}\int_{z}^{\infty} F(x)
(c^{n+1}_t-c^n_t)(dx)\beta_{n+1}(d\theta)\,dz\\
&+ \frac{1}{2}\int_0^{\infty} \partial_x R_n(t,z) \int_0^z \int_0^z \mathds
1_{[z,\infty)}(x+y) \left( K_{n+1}(x,y)-K_{n}(x,y)\right) c^n_t(dy) \\
&\hspace{9.5cm} \times c^n_t(dx) dz \\
&-\frac{1}{2}\int_0^{\infty} \partial_x R_n(t,z)\int_z^{\infty} \int_z^{\infty} \left(
K_{n+1}(x,y)-K_{n}(x,y)\right) c^n_t(dy)\,c^n_t(dx)\,dz \\
&+\int_0^{\infty} \partial_x R_n(t,z)\int_{\Theta}\sum_{i\geq 1}
\int_{z/\theta_i}^{\infty} F(x)\,c^n_t(dx)(\beta_{n+1}-\beta_{n})(d\theta) dz \\
&- \int_0^{\infty} \partial_x R_n(t,z) \int_{\Theta}\int_{z}^{\infty} F(x)
\,c^n_t(dx)(\beta_{n+1}-\beta_{n})(d\theta) dz.
\end{align*}
Thus, after some computations, we obtain 
\begin{equation}\label{DifferentialEn}
\dfrac{d}{dt}\int_0^{\infty} x^{\lambda-1} |E_n(t,x)| dx = I_1^n(t,x) + I_2^n(t,x) + I_3^n(t,x)
+I_4^n(t,x),
\end{equation}
where $I_1^n(t,x)$ and $I_2^n(t,x)$ are respectively the equivalent terms to the coagulation and
fragmentation parts in (\ref{EqIcIf}) and
\begin{align*}
&I_3^n(t,x) =  \frac{1}{2} \int_0^{\infty} \int_0^{\infty} \left( K_{n+1}(x,y)-K_{n}(x,y)\right)
(AR_n(t))(x,y)c^n_t(dy)\,c^n_t(dx) \\
&I_4^n(t,x) = \int_0^{\infty} F(x) \int_{\Theta} (BR_n(t))(\theta,x)
(\beta_{n+1}-\beta_{n})(d\theta) c^n_t(dx),
\end{align*}
which are the terms resulting of the approximation. 

Exactly as in (\ref{unicity}), since the bounds in (\ref{Hyp_Kn}) do not depend on $n$ and that
$\beta_n$ satisfies (\ref{Hyp2_Beta}) uniformly in $n$, we get
\begin{align}\label{Limit2}
& I_1^n(t,x) + I_2^n(t,x) \leq C_1 M_{\lambda}(c^n_t+c^{n+1}_t) \int_0^{\infty}
x^{\lambda-1}|E_n(t,x)|\,dx \\
& \qquad \qquad\qquad\qquad+ C_2 \int_0^{\infty} x^{\lambda-1}|E_n(t,x)|\,dx. \nonumber
\end{align}

Next, since
\begin{align*} 
& K_{n+1}(x,y)-K_{n}(x,y) = \mathds 1_{\{K(x,y)>n+1 \}} + (K(x,y)-n)\mathds 1_{\{n < K(x,y) \leq n+1
\}} \\
& \hspace{3.7cm} \leq \mathds 1_{\{K(x,y)>n\}} \leq \frac{K(x,y)^2}{n^2}
\end{align*}
and using (\ref{R_inequality}), we have
\begin{align}\label{Limit3}
& |I_3^n(t,x)| = \frac{1}{2}\Big|\int_0^{\infty} \int_0^{\infty} \left(
K_{n+1}(x,y)-K_{n}(x,y)\right)(AR_n(t))(x,y) \\
& \hspace{8cm} \times c^n_t(dy)\,c^n_t(dx) \Big|\nonumber\\
&\qquad\leq \frac{1}{2} \int_0^{\infty} \int_0^{\infty} \frac{K(x,y)^2}{n^2} \left| (AR_n(t))(x,y)\right|
c^n_t(dy)\,c^n_t(dx) \nonumber \\
&\qquad \leq  \frac{2^{2\lambda+1}\kappa_0^2}{2\lambda n^2}\int_0^{\infty} \int_0^{\infty} (x \vee
y)^{2\lambda} (x\wedge y)^{\lambda}c^n_t(dy)\,c^n_t(dx) \nonumber\\
&\qquad\leq \frac{C}{n^2} M_{2\lambda}(c^n_t) M_{\lambda}(c^n_t) \,\,  \leq \,\, \frac{1}{n^2}\,C_t,\nonumber
\end{align}
we have used $M_{2\lambda}(c_t) \leq M_{\lambda}(c_t) + M_{2}(c_t)$ together with
(\ref{momentcnbounded}) and (\ref{moment2bounded}).

Finally, since $$\int_{\Theta} (BR_n(t))(\theta,x) \beta_{n}(d\theta) = \int_{\Theta}
(BR_n(t))(\psi_n(\theta),x) \mathds 1_{\{\theta \in\Theta(n)\}} \beta(d\theta),$$ we have
\begin{align}\label{Limit4}
 &\left|I_4^n(t,x) \right|  \\
& = \Big|\int_0^{\infty} F(x) \int_{\Theta} \big\lbrace
\left[(BR_n(t))(\psi_{n+1}(\theta),x) - (BR_n(t))(\psi_n(\theta),x) \right]  \nonumber \\
& \qquad \times \mathds 1_{\Theta(n)\cap\Theta(n+1) } +\, (BR_n(t))(\psi_{n+1}(\theta),x) \mathds 1_{\Theta(n+1)\setminus \Theta(n)} \big\rbrace\beta(d\theta) c^n_t(dx)\Big| \nonumber \\
&\leq \int_0^{\infty} F(x) \int_{\Theta} \left|R_n(t,\theta_{n+1} x)\right|
\mathds 1_{\Theta(n+1)\cap \Theta(n)} \beta(d\theta) c^n_t(dx) \nonumber \\
& \qquad+ \int_0^{\infty} F(x) \int_{\Theta} \Big| \sum_{i=1}^{n+1}R_n(t,\theta_i x)
-R_n(t,x)\Big| \mathds 1_{\Theta(n+1)\setminus \Theta(n)} \beta(d\theta) c^n_t(dx)\nonumber \\
& \leq C \int_0^\infty x^{\lambda} c^n_t(dx) \int_{\Theta}
\theta_{n+1}^{\lambda}\mathds 1_{\{\Theta(n+1) \cap \Theta(n)\}}\beta(d\theta)\nonumber\\
& \qquad+C \int_0^\infty x^{\lambda} c^n_t(dx) \int_{\Theta} \Big[ \sum_{i\geq 2}
\theta_i^{\lambda} + (1-\theta_1)^{\lambda}\Big]\mathds 1_{\{\Theta(n+1) \setminus
\Theta(n)\}}\beta(d\theta)\nonumber\\
& \leq C_t \int_{\Theta} \theta_{n+1}^{\lambda} \beta(d\theta) + C_t
\int_{\Theta} \Big[ \sum_{i\geq 2} \theta_i^{\lambda} + (1-\theta_1)^{\lambda}\Big] \mathds
1_{\{\Theta(n+1) \setminus \Theta(n)\}}\beta(d\theta), \nonumber
\end{align}
we used (\ref{momentcnbounded}). Gathering (\ref{Limit2}), (\ref{Limit3}) and (\ref{Limit4}) in
(\ref{DifferentialEn}) and noting $C(\theta) := \sum_{i\geq 2} \theta_i^{\lambda} +
(1-\theta_1)^{\lambda}$, we obtain
\begin{align*}
& \dfrac{d}{dt}\int_0^{\infty} x^{\lambda-1} |E_n(t,x)| dx \leq C_t M_{\lambda}(c^{in})
\int_0^{\infty} x^{\lambda-1}|E_n(t,x)|\,dx + \frac{1}{n^2}\,C_t\\
& \qquad \qquad \qquad+ C_t \int_{\Theta} \theta_{n+1}^{\lambda} \beta(d\theta) + C_t \int_{\Theta} C(\theta) \mathds
1_{\{\Theta(n+1) \setminus \Theta(n)\}}\beta(d\theta).
\end{align*}
Thus, by the Gronwall Lemma we obtain
\begin{align*}
&\int_0^{\infty} x^{\lambda-1} \big|F^{c^{n+1}_t}(x)-F^{c^{n}_t}(x) \big| dx  \leq 
C_t\Big[\int_0^{\infty} x^{\lambda-1} \big|F^{c^{in,n+1}}(x)-F^{c^{in,n}}(x)\big|dx \\
& \qquad\qquad\qquad \qquad + \frac{1}{n^2}+ \int_{\Theta} \theta_{n+1}^{\lambda} \beta(d\theta) + \int_{\Theta} C(\theta) \mathds 1_{\{\Theta(n+1) \setminus \Theta(n)\}}\beta(d\theta) \Big],
\end{align*} 
for $t\geq0$ and $n\geq1$ and where $C_t$ is a positive constant depending on $\lambda$, $\kappa_0$,
$\kappa_1$, $\kappa_2$, $\kappa_3$, $C^{\lambda}_{\beta}$, $t$ and $c^{in}$. Recalling that
\begin{equation*}
t\mapsto F^{c^{n}_t} \,\,\textrm{belongs to}\,\, \mathcal
C\big([0,\infty);L^1(0,\infty;x^{\lambda-1}dx)\big),
\end{equation*}
for each $n\geq 1$ by Lemma \ref{lemma32}. and Lemma \ref{lemma35}, and since the last three terms
in the right-hand side of the inequality above are the terms of convergent series, we conclude that
$\left(t\mapsto F^{c^{n}_t} \right)_{n\geq 1}$ is a Cauchy sequence in $\mathcal
C\left([0,\infty);L^1(0,\infty;x^{\lambda-1}dx)\right)$ and there is
\begin{equation*}
f \in \mathcal C\big([0,\infty);L^1(0,\infty;x^{\lambda-1}dx)\big)
\end{equation*}
such that
\begin{equation}\label{consq1}
\lim_{n\rightarrow \infty} \sup_{s\in[0,t]} \int_0^{\infty} x^{\lambda-1}
\big|F^{c^{n+1}_s}(x)-f(s,x)\big| dx = 0 \,\,\,\, \textrm{ for each }\,\, t\in[0,\infty).
\end{equation}
As a first consequence of (\ref{consq1}), we obtain that $x\mapsto f(t,x)$ is a non-deacreasing and
non-negative function for each $t\in [0,\infty)$. Furthermore,
\begin{equation}\label{consq2}
\lim_{\varepsilon \rightarrow 0} \sup_{s\in[0,t]} \Big[ \int_0^{\varepsilon} x^{\lambda-1} f(s,x)
dx + \int_{1/\varepsilon}^{\infty}x^{\lambda-1} f(s,x) dx \Big] = 0
\end{equation}
for each $t\in(0,\infty)$ since $f \in \mathcal
C\left([0,\infty);L^1(0,\infty;x^{\lambda-1}dx)\right)$.

We will show that this convergence implies tightness of $(c^ n_t)_{n\geq 1}$ in $\mathcal
M_{\lambda}^+$, uniformly with respect to $s\in[0,t]$. We consider $\varepsilon\in(0,1/4)$, and
since $x\mapsto F^{c^n_s}(x)$ is non-decreasing and $\lambda \in (0,1]$, it follows from Lemma
\ref{lemma32}.:
\begin{equation*}
\int_0^{\varepsilon} x^{\lambda} c^n_t(dx) + \int_{1/\varepsilon}^{\infty}x^{\lambda} c^n_t(dx) \leq
\int_0^{\varepsilon} x^{\lambda-1} F^{c^n_t}(x) dx + \int_{1/(2\varepsilon)}^{\infty}x^{\lambda-1}
F^{c^n_t}(x) dx.
\end{equation*}
The Lebesgue dominated convergence Theorem, (\ref{consq1}) and (\ref{consq2}) give
\begin{equation}\label{tightness}
\lim_{\varepsilon \rightarrow 0} \sup_{n\geq 1} \sup_{s\in[0,t]}\Big[\int_0^{\varepsilon}
x^{\lambda} c^n_t(dx) + \int_{1/\varepsilon}^{\infty}x^{\lambda} c^n_t(dx) \Big] = 0,
\end{equation}
for every $t\in [0,\infty)$. Denoting by $c_t(dx) := -\partial_x f(t,x)$ the derivative with respect
to $x$ of $f$ in the sense of distributions for $t \in (0,\infty)$, we deduce from
(\ref{momentcnbounded}), (\ref{consq1}) and (\ref{tightness}) that $c_t(dx) \in \mathcal
M_{\lambda}^+$ with $M_{\lambda}(c_t) \leq e^{ \kappa_2 C^{\lambda}_{\beta} t} M_{\lambda}(c^{in})$.

Consider now $\phi \in \mathcal C^1_ c((0,\infty))$ and recall that $|\phi'(x)|\leq C x^{\lambda-1}$
for some positive constant $C$. On the one hand, the time continuity of $f$ implies that
\begin{equation*}
t\mapsto \int_0^{\infty} \phi(x)c_t(dx) = \int_0^{\infty} \phi'(x) f(t,x)dx
\end{equation*}
is continuous on $[0,\infty)$. On the other hand, the convergence (\ref{consq1}) entails
\begin{align}\label{consq3}
&\lim_{n \rightarrow \infty} \sup_{s\in[0,t]}\Big|\int_0^{\infty} \phi(x) (c^n_s 
- c_s)(dx) \Big| \\
& \qquad = \lim_{n \rightarrow \infty} \sup_{s\in[0,t]}\Big|\int_0^{\infty}
 \phi'(x) \left(F^{c^n_s}(x) - F^{c_s}(x)\right)dx \Big|\nonumber\\
& \qquad \leq  \lim_{n \rightarrow \infty} \sup_{s\in[0,t]}\Big| C \int_0^{\infty}
 x^{\lambda-1} \left(F^{c^n_s}(x) - F^{c_s}(x)\right)dx \Big| = 0, \nonumber
\end{align}
for every $t\geq0$. We then infer from (\ref{tightness}), (\ref{consq3}), 
Lemma \ref{lemma31}., (\ref{AuxIf}) and a density argument that for every 
$\phi \in \mathcal H_{\lambda}$, the map $t\mapsto \int_0^{\infty} \phi(x) c_t(dx)$ is continuous
and
\begin{align*}
& \lim_{n \rightarrow \infty} \sup_{s\in[0,t]}\Big|\int_0^{\infty}
 \phi(x) (c^n_s - c_s)(dx) \Big| = 0 ,\\
& \lim_{n \rightarrow \infty} \sup_{s\in[0,t]} \Big| \frac{1}{2}\int_0^{\infty}\int_0^{\infty}
 (A\phi)(x,y) ] K(x,y) (c^n_s(dx)c^n_s(dy) -c_s(dx)c_s(dy) )\\
& \qquad  \qquad \qquad + \int_0^{\infty}F(x)\int_{\Theta} (B\phi)(\theta,x)  \beta(d\theta) (c^n_s-c_s)(dx) 
 \Big|  =  0.
\end{align*}
We may thus pass to the limit as $n\rightarrow \infty$ in the integrated form of
 (\ref{weakeq2}) for $(c^n_t)_{t\geq 0}$ and deduce that for all $t\geq 0$ and 
 $\phi \in\mathcal H_{\lambda}$, we have
\begin{align}\label{weakeq2final}
& \int_0^{\infty}\phi(x) c_t(x)\, dx  = \int_0^{\infty}\phi(x) c^{in}(x)\, dx  \\
& \qquad \qquad +\frac{1}{2}\int_0^{\infty}\int_0^{\infty} [\phi(x+y) - \phi(x) - \phi(y)] 
K(x,y) c_t(dx) c_t(dy) \nonumber\\
& \qquad \qquad + \int_0^{\infty}\int_{\Theta} \Big [\sum_{i=1}^{\infty}\phi(\theta_i x)
 - \phi(x)\Big] F(x) \beta(d\theta) c_t(dx). \nonumber 
\end{align}
Classical arguments then allows us to differentiate (\ref{weakeq2final})
 with respect to time and conclude that $(c^n_t)_{t\geq 0}$ 
 is a $(c^{in},K,F,\beta,\lambda)$-weak solution to (\ref{forteq2}). 
\end{proof}

\noindent \textbf{Existence and uniqueness for 
$c^{in} \in \mathcal M_{\lambda}^+$.-}   

We have shown existence for 
$c^{in} \in \mathcal M_{\lambda}^+ \cap \mathcal M_{2}^+$.
 Now we are going to extend the previous result to an initial condition only in 
 $\mathcal M_{\lambda}^+$. For this, we consider $(a_n)_{n\geq 1}$ and 
 $(A_n)_{n\geq 1}$ two sequences in $\mathbb R^+$ such that $a_n$ is
  non-increasing and converging to $0$ and $A_n$ non-decreasing and
   tending to $+\infty$ with $0 < a_0 \leq A_0$. We set $B_n = [a_n,A_n]$ and define
\begin{equation*}
c^{in,n}(dx) := c^{in}|_{B_n}(dx),
\end{equation*}
note that trivially we have $M_{2}(c^{in,n}) < \infty$. Next, we call 
$(\tilde c^n_t)_{t\geq 1}$ the $(c^{in,n},K,F,\beta,\lambda)$-weak 
solution to (\ref{forteq2}) constructed in the previous section.  

Owing to Proposition \ref{Prop_Uniqueness}. and (\ref{unicity}), 
  we have for $t\geq0$ and $n\geq1$ 
\begin{align*}
\int_0^{\infty} x^{\lambda-1} \big|F^{\tilde c^{n+1}_t}(x)
-F^{\tilde c^{n}_t}(x)\big| dx \leq e^{C_t} \int_0^{\infty} x^{\lambda-1}
 \big|F^{ c^{in,n+1}}(x)-F^{ c^{in,n}}(x)\big| dx,
\end{align*}
  Next, we have
\begin{align*}
& \int_0^{\infty} x^{\lambda-1} \big|F^{ c^{in,n+1}}(x)-F^{ c^{in,n}}(x)\big|\,dx \\
&\qquad =   \int_0^{+\infty}x^{\lambda-1} \Big| \int_0^{+\infty}
\mathds 1_{[x,+\infty)}(y) \left( c^{in}|_{B_n} - c^{in}|_{B_{n+1}}\right)(dy) \Big| dx \nonumber
\\
&\qquad =   \int_0^{+\infty}x^{\lambda-1}  \int_0^{+\infty}  \mathds 1_{[x,+\infty)}(y) \left( \mathds 1_{[a_{n+1},a_n)}(y) +   \mathds 1_{[A_n,A_{n+1})}(y)\right) c^{in}(dy) dx, \nonumber 
\end{align*}
\begin{sloppypar}
\noindent note that since $\sum_{n\geq 0} \left[ \mathds 1_{[a_{n+1},a_n)}(y)
 + \mathds 1_{[A_n,A_{n+1})}(y)\right] \leq  \mathds 1_{\mathbb R^+}(y)$
  the term in the right-hand of the last inequality is summable.
   We conclude that $\left(t\mapsto F^{\tilde c^{n}_t} \right)_{n\geq 1}$
is a Cauchy sequence in $\mathcal C\big([0,\infty);L^1(0,\infty;x^{\lambda-1}dx)\big)$ and there
is
\end{sloppypar}
\begin{equation*}
f \in \mathcal C\big([0,\infty);L^1(0,\infty;x^{\lambda-1}dx)\big),
\end{equation*}
such that
\begin{equation*}
\lim_{n\rightarrow \infty} \sup_{s\in[0,t]} \int_0^{\infty} x^{\lambda-1}
\big|F^{\tilde c^{n+1}_s}(x)-f(s,x)\big| dx = 0 \,\,\,\, \textrm{ for each }\,\, t\in[0,\infty).\end{equation*}
and we conclude using the same arguments as in the previous case, setting 
$c_t := -\partial_x f(t,x)$ in the sense of distributions, that 
$(c_t)_{t\geq 0}$ is a $(c^{in},K,F,\beta,\lambda)$-weak solution to
 (\ref{forteq2}) in the sense of Definition \ref{Def_solution}. 

This completes the proof of Theorem \ref{Theorem}.
\end{proof}

I would like to thank my Ph.D. advisor Prof. Nicolas Fournier for his insightful
 comments and advices during the preparation of this work. 
 I would like also to thank B\'en\'edicte Haas for the lecture and her comments on this work.

\end{document}